\documentclass[amssymb,twoside,11pt]{article}
\usepackage{latexsym,amsmath,graphicx}
\usepackage{amsfonts}
\newtheorem{theorem}{Theorem}[section]
\newtheorem{lemma}[theorem]{{\bf Lemma}}

\newtheorem{cor}[theorem]{{\bf Corollary}}
\newtheorem{rem}[theorem]{{\bf Remark}}
\newtheorem{ex}[theorem]{{\bf Example}}
\newtheorem{definition}{Definition}[section]
\numberwithin{equation}{section}
\newenvironment{proof}{\indent{\em Proof:}}{\quad \hfill
$\Box$\vspace*{2ex}}

\font\Bbb=msbm10 at 12pt

\newcommand{\R}{\mbox{\Bbb R}}

\begin{document}
\setcounter{page}{1}
\begin{center}
\vspace{0.4cm} {\large{\bf Analysis of  Impulsive $\varphi$--Hilfer Fractional Differential Equations}}\\
\vspace{0.5cm}
Kishor D. Kucche $^{1}$ \\
kdkucche@gmail.com \\

\vspace{0.35cm}
Jyoti P. Kharade  $^{2}$\\
jyoti.thorwe@gmail.com\\

\vspace{0.35cm}
$^{1,2}$ Department of Mathematics, Shivaji University, Kolhapur-416 004, Maharashtra, India.
\end{center}
\def\baselinestretch{1.0}\small\normalsize
\begin{abstract}
This paper is concerned with the  existence and uniqueness,  and  Ulam--Hyers stabilities of solutions of nonlinear impulsive $\varphi$--Hilfer fractional differential equations. Further, we investigate the dependence of the solution on the initial conditions, order of derivative and the functions involved in the equations.  The outcomes are acquired in the space of weighted piecewise continuous functions by means of fixed point theorems and the generalized version of Gronwall inequality.
\end{abstract}
\def\baselinestretch{1.5}

\allowdisplaybreaks
\section{Introduction}

The qualitative theory of impulsive fractional differential equations (FDEs) is of extraordinary significance in view of its applications in exhibiting numerous characteristic of physical phenomena which are appearing in the field of medicine,   biology, mechanics  and electrical engineering. Consequently numerous scientists \cite{fec}--\cite{Ali} engaged with working on impulsive FDEs and examining the existence and uniqueness, dependence and various kinds of stabilities of solution.

On the other hand, the representation formula for the solution of impulsive FDEs of the form:
\begin{equation} \label{imp01}
\begin{cases}
^c D^{\varrho}u(t)=f(t, u(t)),~t \in [0,T], t\neq t_{k},\\
\Delta u|_{t=t_{k}}= I_{k}(y(t_{k}^-)) \\
 u(0)=u_{0} \in \mathbb{R},
\end{cases}
\end{equation}
have different approaches \cite{fec,WAZN}.  Consequently  analysis of impulsive FDEs have been done with different representation formula of the solution.  Wang et al. \cite{Wang1,Wang22,Wang11,Wang2} have analyzed nonlinear impulsive differential equations with Caputo fractional derivative for existence, uniqueness and data dependence of solutions via generalized singular Gronwall inequalities.  Benchohra et al. \cite{Benchohra1, Benchohra,  Benchohra2, Benchohra3} investigated the sufficient conditions for the existence of solutions for different kinds of impulsive FDEs. Mophou \cite{Mophou} investigated the existence and uniqueness of a mild solution to impulsive
semilinear  FDEs. Zhang and Wang \cite{Zhang} examined
the existence of solutions for an anti-periodic boundary value
problem of nonlinear impulsive FDEs via fixed point theorems.
For other interesting work on analysis of different class of impulsive fractional differential equations we refer the reader to \cite{AS}--\cite{WAZ}.

On the other hand, Harrat et al.\cite{Harrat}, utilizing the tools of  fixed point technique, semigroup theory and multivalued analysis, investigated the  solvability and optimal controls of an impulsive nonlinear   delay evolution inclusion in Banach spaces with Hilfer fractional derivative. Harikrishnan et al.\cite{Harikrishnan}, established existence and stability of solutions
for impulsive Hilfer FDEs using fixed point theorem of  Banach and Schaefer. Ahmed  et al. \cite{Ahmed} examined existence and approximate controllability for Sobolev-type impulsive FDEs involving the Hilfer  derivative.

However, as talked about in \cite{Fernandez}, numerous definitions of fractional derivatives and integrals have been proposed in the literature. One can see that existence, uniqueness and many other essential qualitative properties of solution for FDEs have been demonstrated for similar type of nonlinear FDEs with different fractional order derivative operators. Hence it is have to research nonlinear FDEs with more general fractional operators which incorporates all the specific fractional derivative operators including Caputo  and Riemann-Liouville derivatives and all other well known fractional derivative operators such as Hadamard derivative, Katugampola derivative, Hilfer derivative, Chen derivative, Prabhakar derivative, Erdlyi-Kober derivative, Riesz derivative, Feller derivative, Weyl derivative, Cassar derivative etc. Hence it is imperative to analyze the impulsive FDEs with broad class of  fractional derivative operator  that incorporate various definitions of well known fractional derivatives.

The definition of the Riemann-Liouville fractional integral have been
extended to  a fractional integral of a function with respect to the another function viz. $ \varphi $-Riemann-Liouville fractional integral \cite[Chapter 2]{Kilbas}. Using  concept of generalized fractional integral,  the  Riemann-Liouville and Caputo version of fractional derivative have been introduced namely $ \varphi $-Riemann-Liouville fractional derivative \cite[Chapter 2]{Kilbas} and the  $\varphi$-Caputo fractional derivative \cite{Almeida}.  Properties of these generalized fractional derivatives one can find in \cite{Kilbas,Almeida}.  Jarad  and coauthors \cite{Ameen}--\cite{Jarad1} analyzed distinctive class of FDEs for existence, uniqueness and the Ulam-Hyers stabilities of solution. Following the technique of \cite{Kilbas, Almeida}, Sousa and Olivera \cite{Sousa1,Sousa2} presented a Hilfer version of fractional derivative viz. $\varphi$--Hilfer fractional derivative  which incorporate the Hilfer fractional derivative  \cite{Hilfer} as well as includes a wide class of well known fractional derivatives including most widely used Caputo  and Riemann-Liouville derivative. 

The list of all possible fractional derivatives which are the particular cases of $\varphi$--Hilfer fractional derivatives have been provided in \cite{Sousa1}.

 In \cite{JVCO}, Sousa et al. established  generalized Gronwall inequality involving $\varphi$-Riemann-Liouville fractional integral and utilized it to investigate the qualitative properties of solutions such as uniqueness, continuous dependence of solution on different data for the Cauchy-type problem with $ \varphi $ -Hilfer fractional derivative. Kucche et al. \cite{KKAM} investigated the existence and uniqueness, dependence of solution for $\varphi$--Hilfer FDEs in the weighted space of functions through Weissinger fixed point theorem. Further derived representation formula for the solution of linear Cauchy problem for $\varphi$--Hilfer FDEs obtained by using Picard’s successive approximation method.  
 
In \cite{KKJK},  Kucche et al. investigated the  formula for the solution of 
 nonlinear $\varphi$-Hilfer impulsive FDEs and established the
 existence and uniqueness results. It is  proved that the obtained formula for the solution  of nonlinear $\varphi$-Hilfer impulsive FDEs includes  the formula for the solution of  impulsive FDEs involving  Riemann-Liouville and Caputo derivatives.  Sousa et al. \cite{SousaKucche}  derived sufficient conditions to ensure existence and uniqueness of solutions and $ \delta $--Ulam--Hyers--Rassias stability of an impulsive FDEs involving $ \varphi $-Hilfer fractional derivative. Liu et al. \cite{Liu}  presented existence, uniqueness, and Ulam--Hyers--Mittag--Leffler stability of solutions to a class of $ \varphi $-Hilfer fractional-order delay differential equations through Picard operator thoery and a generalized Gronwall inequality.

In the present paper, we consider the nonlinear impulsive $\varphi$-Hilfer fractional differential equation ($\varphi$-HFDE) with initial condition of the form:
\begin{align}\label{p1}
\begin{cases}
 ^H \mathbf{D}^{\varrho,\, \nu; \, \varphi}_{a^+}u(t)=f(t, u(t)),~t \in J=[a,T]-\{t_1, t_2,\cdots ,t_m\}\\
\Delta \mathbf{I}_{a^+}^{1-\sigma; \, \varphi}u(t_k)= \mathcal{J}_k(u(t_k^-)),  ~k = 1,2,\cdots,m, \\
 \mathbf{I}_{a^+}^{1-\sigma; \, \varphi}u(a)=u_a \in \mathbb{R},  ~\sigma=\varrho+\nu-\varrho \nu, 
\end{cases}
\end{align}
where $\varphi\in C^{1}(\mathcal{I},\mathbb{R})$ be an increasing function with $\varphi'(x)\neq 0$, for all $x\in \mathcal{I}$, ~$^H \mathbf{D}^{\varrho, \, \nu; \, \varphi}_{a^+}(\cdot)$ is the $\varphi$-Hilfer fractional derivative \cite{Sousa1} of order $\varrho~ (0<\varrho<1)$ and type $\nu ~(0\leq\nu\leq 1)$ defined by 
$$
^H \mathbf{D}^{\varrho, \, \nu; \, \varphi}_{a^+}f(t)= \mathbf{I}_{a^+}^{\nu ({1-\varrho});\, \varphi} \left(\frac{1}{{\varphi}^{'}(t)}\frac{d}{dt}\right)^{'}\mathbf{I}_{a^+}^{(1-\nu)(1-\varrho);\, \varphi} f(t),
$$
and $\mathbf{I}_{a^+}^{1-\sigma; \, \varphi}$ is left sided $\varphi$-Riemann Liouville fractional integration operator \cite{Sousa1}, which defined for any $\varrho>0$ 
$$
\mathbf{I}_{a+}^{\varrho ;\, \varphi }f\left( t\right) :=\frac{1}{\Gamma \left( \varrho
\right) }\int_{a}^{t}\varphi ^{\prime }\left(\sigma \right) \left( \varphi \left(
t\right) -\varphi \left( \sigma \right) \right) ^{\varrho-1}f\left( \sigma \right) d\sigma.
$$
Let $ ~a=t_0< t_1 < t_2 < \cdots < t_m < t_{m+1}=T$,~$
\Delta \mathbf{I}_{a^+}^{1-\sigma; \, \varphi}u(t_k)= \mathbf{I}_{a^+}^{1-\sigma; \, \varphi}u(t_k^+)- \mathbf{I}_{a^+}^{1-\sigma; \, \varphi}u(t_k^-) 
$, 
~$
 \mathbf{I}_{a^+}^{1-\sigma; \, \varphi}u(t_k^+) = \lim_{\epsilon\to 0^+} \mathbf{I}_{a^+}^{1-\sigma; \, \varphi}u(t_k + \epsilon)$ 
~~\mbox{and}~~
$\mathbf{I}_{a^+}^{1-\sigma; \, \varphi}u(t_k^-) = \lim_{\epsilon\to 0^-} \mathbf{I}_{a^+}^{1-\sigma; \, \varphi}u(t_k + \epsilon)
$.  The functions $f:(a,T]\times \mathbb{R} \to \mathbb{R}$ and $\mathcal{J}_k:\mathbb{R}\to \mathbb{R}$ are appropriate functions specified latter.

The motivation for the  work presented in the present paper is originated from  \cite{Wang1,Wang2,KKJK}. Main objective of the present paper is to establish the existence results for the impulsive $\varphi$-Hilfer fractional differential equation (impulsive  $\varphi$-HFDE ) with initial condition via Schaefer fixed point theorem. Additionally, by utilizing the generalized version of Gronwall inequality, we  examine the uniqueness of solution,  dependence of the solution on the initial conditions, order of the $ \varphi $-Hilfer derivative and the functions involved in the equations. Further, we investigate the  Ulam--Hyers and Ulam--Hyers--Rassias stabilities via generalized version of Gronwall inequality.

The remainder of this paper is organized as follows. In Section 2, existence results are exhibited. In Section 3, the dependence of the solution associated with initial conditions, the functions involved in the equations and the order of $ \varphi $-Hilfer derivative have been   examined. Section 4, deals with the Ulam--Hyers and Ulam--Hyers--Rassias stabilities for \eqref{p1}. Finally in section 5, we give an example to illustrate the obtained results.

\section{Preliminaries} \label{preliminaries}
Consider the weighted space \cite{Sousa1} defined by
 $$
C_{1-\sigma;\varphi}(\mathcal{I})=\left\{u:(a,b]\to\mathbb{R} : ~(\varphi(t)-\varphi(a))^{1-\sigma}u(t)\in C(\mathcal{I})\right\},0< \sigma\leq 1.
$$
Define the weighted space of piecewise continuous functions as
\begin{align*}
\mathcal{PC}_{1-\sigma; \,  \varphi}(\mathcal{I},\mathbb{R}) =\{& u:(a,b]\to\mathbb{R} :u\in C_{1-\sigma;\varphi}((t_k,t_{k+1}],\mathbb{R}),k=0,1,2,\cdots, m,\\
& \mathbf{I}_{a^+}^{1-\sigma; \, \varphi} \, u(t_k^+), ~ \mathbf{I}_{a^+}^{1-\sigma; \, \varphi}\, u(t_k^-) ~\mbox{exists and}  ~\mathbf{I}_{a^+}^{1-\sigma; \, \varphi}\, u(t_k^-)= \mathbf{I}_{a^+}^{1-\sigma; \, \varphi} \, u(t_k) \\
& ~\mbox{for} ~ \, k = 1,2,\cdots, m \}
\end{align*}
Clearly, $\mathcal{PC}_{1-\sigma; \, \varphi}(\mathcal{I},\mathbb{R})$ is a Banach space with the norm
$$
\|u\|_{\mathcal{PC}_{1-\sigma; \,  \varphi}(\mathcal{I},\mathbb{R})} = \sup_ {t\in \mathcal{I}} \left|(\varphi(t)-\varphi(a))^{1-\sigma}u(t)\right|.
$$
Note that for $\sigma =1$, we get 
$ \mathcal{PC}_{0; \, \varphi}(\mathcal{I},\mathbb{R})=PC(\mathcal{I},\mathbb{R})$ 
a particular case of the space $\mathcal{PC}_{1-\sigma; \,  \varphi}(\mathcal{I},\mathbb{R})$, whose  details are given in \cite{Benchohra,Wang, Bai}.

\begin{lemma}[$\mathcal{PC}_{1-\sigma; \, \varphi}$ type Arzela-–Ascoli Theorem, \cite{KKJK}] \label{pc}

Let $\mathcal{X}$ be a Banach space and $\mathcal{W}_{1-\sigma;\, \varphi} \subset \mathcal{PC}_{1-\sigma; \, \varphi}(\mathcal{J},\mathcal{X}).$ If the following conditions are satisfied:
\begin{itemize} 
\item [{\rm(a)}]$\mathcal{W}_{1-\sigma;\, \varphi}$ is uniformly bounded subset of $\mathcal{PC}_{1-\sigma; \, \varphi}(\mathcal{J}, \mathcal{X})$;
\item [{\rm(b)}] $\mathcal{W}_{1-\sigma;\, \varphi}$ is equicontinuous in $(t_k, t_{k+1}), k = 0, 1, 2,\cdots , m, where~ t_0 = a, t_{m+1} = T$ ;
\item [{\rm(c)}]$\mathcal{W}_{1-\sigma;\, \varphi}(t) = \{u(t): u \in \mathcal{W}_{1-\sigma;\, \varphi}, ~t \in \mathcal{J} - {t_1, \cdots , t_m}\}, \mathcal{W}_{1-\sigma;\, \varphi}(t_k^+ ) = \{u(t_k^+ ): u \in  \mathcal{W}_{1-\sigma;\, \varphi}\}~ \text{and}~ \mathcal{W}_{1-\sigma;\, \varphi}(t_k^- ) = \{u(t_k^- ): u \in  \mathcal{W}_{1-\sigma;\, \varphi}\}$ are relatively compact subsets of X, 
\end{itemize}
then $\mathcal{W}_{1-\sigma;\, \varphi}$ is a relatively compact subset of $\mathcal{PC}_{1-\sigma; \, \varphi}(\mathcal{J}, X)$.
\end{lemma}

\begin{theorem}[Schaefer, \cite{zhou}]\label{kf}
Let $\mathcal{F}:C(J,\mathbb{R})\to C(J,\mathbb{R}) $ be a completely continuous operator. If the set 
$$
\mathbf{G}(\mathcal{F})=\left\lbrace u \in C(J,\mathbb{R}): u=\lambda \, \mathcal{F}(u), \text{for some}\, \lambda \in (0,1) \right\rbrace  
$$
is bounded, then $\mathcal{F}$ has at least one fixed point.
\end{theorem}
\begin{lemma}[\cite{Jose}]\label{JI}
Let $ \mathcal{U}\in {PC}_{1-\sigma;\, \varphi}\left( J,\,\mathbb{R}\right) $ satisfying the following inequality
$$
\mathcal{U}(t)\leq \mathcal{V}(t)+ \mathbf{g}(t)\,\int_{a}^{t}\varphi^{'}(s)(\varphi(t)-\varphi(s))^{\varrho-1}\,\mathcal{U}(s)\, ds + \sum_{a<t_k<t}\beta_k \mathcal{U}(t_k^-),\, t>a,
$$
where $\mathbf{g}$ is a continuous function, $\mathcal{V}\in {PC}_{1-\sigma;\, \varphi}\left( J,\,\mathbb{R}\right) $ is non-negative, $\beta_k > 0$ for $k=1,2,\cdots,m,$ then we have
\begin{align*}
\mathcal{U}(t)\leq \mathcal{V}(t) &\left[\prod_{i=1}^{k}\left\lbrace 1+\beta_i E_\varrho \left(\mathbf{g}(t)\Gamma(\varrho)(\varphi(t_i)-\varphi(a))^\varrho \right)  \right\rbrace  \right]\\
&\times E_\varrho\left(\mathbf{g}(t)\Gamma(\varrho)(\varphi(t)-\varphi(a))^\varrho \right),\,\, t\in (t_k,t_{k+1}]
\end{align*} 
\end{lemma}

To analyze the impulsive $\varphi$-HFDE \eqref{p1}, we utilize its equivalent fractional integral given in the following Lemma.

\begin{lemma} [\cite{KKJK}]\label{fie}
Let $h: J\to \mathbb{R}$ be a continuous function.  Then a function $u \in {PC}_{1-\sigma;\, \varphi}\left( J,\,\mathbb{R}\right)$ is a solution of impulsive $\varphi$--HFDE  \eqref{p1}
if and only if u is a solution of the following fractional integral equation
\begin{equation}\label{e14}
u(t) =
\begin{cases}
\frac{(\varphi(t)-\varphi(a))^{\sigma-1}}{\Gamma(\sigma)}\, u_a +\mathbf{I}_{a^+}^{\varrho; \, \varphi}h(t),~ \text{$t \in [a,t_1], $}\\
 \frac{(\varphi(t)-\varphi(a))^{\sigma-1}}{\Gamma(\sigma)}\, \left(u_a +\sum_{i=1}^{k}\mathcal{J}_i(u(t_i^-))\right)+\mathbf{I}_{a^+}^{\varrho; \, \varphi}h(t), ~\text{ $t \in (t_k,t_{k+1}],~ k=1,2,\cdots,m $}.  
\end{cases}
\end{equation}
\end{lemma}

\section{Existence results}

In this section, we derive the existence of solution to the problem \eqref{p1} by means of Schaefer's fixed point theorem.
\begin{theorem}\label{exi}
Assume that:
\begin{itemize}
\item[($H_1$)] The function $f:(a,T]\times \mathbb{R} \to \mathbb{R}$ is continuous and satisfy the following conditions:
\begin{itemize}
\item [(i)] $f(\cdot,u(\cdot))\in {\mathcal{PC}}_{1-\sigma;\, \varphi}\left( J,\,\mathbb{R}\right)$ for any $u\in {\mathcal{PC}}_{1-\sigma;\, \varphi}\left( J,\,\mathbb{R}\right),$

\item [(ii)] $\left| f(t,u(t))-f(t,v(t))\right|\leq (\varphi(t)-\varphi(a))^{1-\sigma}\,|u(t)-v(t)| 
$, $t\in J$,~ $u,v \in {\mathcal{PC}}_{{1-\sigma};\,\varphi} \left( J,\,\mathbb{R}\right),$
\end{itemize}
\item[($H_2$)] The functions $\mathcal{J}_k:\mathbb{R}\to \mathbb{R}$, ($k=1,2,\cdots,m.$) satisfy the condition

\begin{itemize}
\item [(i)] $
\left| \mathcal{J}_k(u(t_k^-))-\mathcal{J}_k(v(t_k^-))\right|\leq (\varphi(t_k^-)-\varphi(a))^{1-\sigma}\,|u(t_k^-)-v(t_k^-)|$,~ $u\in {\mathcal{PC}}_{1-\sigma;\, \varphi}\left( \mathcal{J},\,\mathbb{R}\right),$
\item [(ii)] there exist $\zeta_k > 0$ such that $|\mathcal{J}_k(u(t_k^-))| \leq \zeta_k$, ~$u\in {\mathcal{PC}}_{1-\sigma;\, \varphi}\left( \mathcal{J},\,\mathbb{R}\right)$.
\end{itemize}
\end{itemize}
Then, the problem \eqref{p1} has at least one solution in the space ${\mathcal{PC}}_{{1-\sigma};\,\varphi}\left( J,\,\mathbb{R}\right).$
\end{theorem}
\begin{proof}
Consider the operator $\mathbb{F}$ defined on ${\mathcal{PC}}_{{1-\sigma};\,\varphi}\left( J,\,\mathbb{R}\right)$ by
\begin{equation}\label{41}
(\mathbb{F}u)(t) =
 \frac{(\varphi(t)-\varphi(a))^{\sigma-1}}{\Gamma(\sigma)}\, \left(u_a +\sum_{a<t_k<t}\mathcal{J}_k(u(t_k^-))\right)+\mathbf{I}_{a^+}^{\varrho; \, \varphi}f(t,u(t)), ~\text{ $t \in J$ }.
\end{equation}
Then as proved in \cite{KKJK}, $\mathbb{F}$ is mapping from
 $ {\mathcal{PC}}_{{1-\sigma};\,\varphi}\left( J,\,\mathbb{R}\right)$ to itself. Further, our problem of finding the solution to \eqref{p1} is reduced to find a fixed point to the operator $\mathbb{F}$. Utilizing Schaefer fixed point theorem, we prove  that the operator $\mathbb{F}$ has fixed point, and the proof of same given in following four steps. 
 
 Step 1: $\mathbb{F}$ is continuous.\\
Let $\{u_n\}$ be a sequence such that $u_n\to u \, \text{in} ~{\mathcal{PC}}_{{1-\sigma};\,\varphi}\left( J,\,\mathbb{R}\right).$ Then for each $t\in J,$ 
\begin{align*}
&(\varphi(t)-\varphi(a))^{1-\sigma}\left| (\mathbb{F} u_n)(t)- (\mathbb{F}u)(t)\right| \\
&
=\left| \frac{1}{\Gamma(\sigma)}\, \left(u_a +\sum_{a<t_k<t}\mathcal{J}_k(u_n (t_k^-))\right)+ (\varphi(t)-\varphi(a))^{1-\sigma}\,\,\mathbf{I}_{a^+}^{\varrho; \, \varphi}f(t,u_n (t))\right.\\
&\qquad \left. - \frac{1}{\Gamma(\sigma)}\,\, \left(u_a +\sum_{a<t_k<t}\mathcal{J}_k(u(t_k^-))\right)-(\varphi(t)-\varphi(a))^{1-\sigma}\, \mathbf{I}_{a^+}^{\varrho; \, \varphi}f(t,u(t)) \right|\\
& \leq \frac{1}{\Gamma(\sigma)}\sum_{a<t_k<t}\left| \mathcal{J}_k(u_n (t_k^-))-\mathcal{J}_k(u(t_k^-))\right| \\
&\qquad + (\varphi(t)-\varphi(a))^{1-\sigma}\,\,\left|\mathbf{I}_{a^+}^{\varrho; \, \varphi}\left(f(t,u_n(t))-f(t,u(t)) \right)  \right|
\end{align*}
Since, $f$ and $\mathcal{J}_k, (k=1,\cdots, n)$ are continuous functions, we have
$$
\left| \mathcal{J}_k(u_n (t_k^-))-\mathcal{J}_k(u(t_k^-))\right|\to 0
$$
and
$$
\left|\left(f(t,u_n(t))-f(t,u(t)) \right) \right| \to 0 
$$
as $n\to \infty$.
Therefore,
$$
(\varphi(t)-\varphi(a))^{1-\sigma} \,\, \left| (\mathbb{F} u_n)(t)- (\mathbb{F}u)(t)\right| \to 0 \, \text{as}\, n\to \infty, 
$$ 
which gives

$$
 \left\| \mathbb{F} u_n- \mathbb{F}u\right\|_ {{\mathcal{PC}}_{{1-\sigma};\,\varphi}\left( J,\,\mathbb{R}\right)}\to 0 \, \text{as}\, n\to \infty. 
$$
This proves the operator $\mathbb{F}:{\mathcal{PC}}_{{1-\sigma};\,\varphi}\left( J,\,\mathbb{R}\right) \to {\mathcal{PC}}_{{1-\sigma};\,\varphi}\left( J,\,\mathbb{R}\right)$ 
is continuous.

Step 2: $\mathbb{F}$ maps bounded sets into bounded sets.\\
To prove this we shall show that for $\delta>0$ there exists $\eta >0$ such that for any $u\in \mathbb{D}_\delta = \left\lbrace u\in {\mathcal{PC}}_{{1-\sigma};\,\varphi}\left( J,\,\mathbb{R}\right): \left\|u \right\|_{{\mathcal{PC}}_{{1-\sigma};\,\varphi}\left( J,\,\mathbb{R}\right)} \leq \delta  \right\rbrace $ one has $\left\| \mathbb{F}u\right\|_{{\mathcal{PC}}_{{1-\sigma};\,\varphi}\left( J,\,\mathbb{R}\right)}\leq \eta.$ 
Let $\mathcal{M}^* = \sup_{s\in J} |f(s,0)|$ and $\zeta =\sum_{k=1}^{m}\zeta_k $. Then, using the hypotheses $(H_1)(ii)$ and $(H_2)(ii)$, for each $t \in J$, we have
\begin{align*}
&(\varphi(t)-\varphi(a))^{1-\sigma}|(\mathbb{F}u)(t)|\\
& \leq \frac{|u_a|}{\Gamma(\sigma)} \,  + \frac{1}{\Gamma(\sigma)}\, \sum_{a<t_k<t}\left| \mathcal{J}_k(u(t_k^-))\right| + (\varphi(t)-\varphi(a))^{1-\sigma}\left|\mathbf{I}_{a^+}^{\varrho; \, \varphi}f(t,u(t)) \right|\\
& \leq \frac{|u_a|}{\Gamma(\sigma)} \,  + \frac{1}{\Gamma(\sigma)}\, \sum_{a<t_k<t} \zeta_k +\frac{(\varphi(t)-\varphi(a))^{1-\sigma}}{\Gamma(\varrho)} \int_{a}^{t}\varphi^{'}(s)(\varphi(t)-\varphi(s))^{\varrho-1}\,\left|f(s, u(s))- f(s,0) \right| \, ds \\
&\qquad + \frac{(\varphi(t)-\varphi(a))^{1-\sigma}}{\Gamma(\varrho)} \int_{a}^{t}\varphi^{'}(s)(\varphi(t)-\varphi(s))^{\varrho-1}\,\left|f(s,0) \right| \, ds\\
& \leq \frac{|u_a|}{\Gamma(\sigma)} + \frac{\zeta}{\Gamma(\sigma)}+ \frac{(\varphi(t)-\varphi(a))^{1-\sigma}}{\Gamma(\varrho)} \int_{a}^{t}\varphi^{'}(s)(\varphi(t)-\varphi(s))^{\varrho-1} (\varphi(s)-\varphi(a))^{1-\sigma} \,\left|u(s) \right| \, ds \\
&\qquad + \frac{\mathcal{M}^* (\varphi(t)-\varphi(a))^{1-\sigma}}{\Gamma(\varrho)} \int_{a}^{t}\varphi^{'}(s)(\varphi(t)-\varphi(s))^{\varrho-1}\, ds \\
& \leq \frac{|u_a|+\zeta}{\Gamma(\sigma)} + \frac{(\varphi(t)-\varphi(a))^{1-\sigma}}{\Gamma(\varrho)} \left\|u \right\|_{{\mathcal{PC}}_{{1-\sigma};\,\varphi}\left( J,\,\mathbb{R}\right)}\int_{a}^{t}\varphi^{'}(s)(\varphi(t)-\varphi(s))^{\varrho-1}\, ds \\
&\qquad + \frac{\mathcal{M}^* (\varphi(t)-\varphi(a))^{1-\sigma}}{\Gamma(\varrho)} \int_{a}^{t}\varphi^{'}(s)(\varphi(t)-\varphi(s))^{\varrho-1}\, ds \\
&= \frac{|u_a|+\zeta}{\Gamma(\sigma)} + \frac{(\varphi(t)-\varphi(a))^{1-\sigma+\varrho}}{\Gamma(\varrho+1)} \left( \left\|u \right\|_{{\mathcal{PC}}_{{1-\sigma};\,\varphi}\left( J,\,\mathbb{R}\right)}+ \mathcal{M}^* \right)\\
& \leq \frac{|u_a|+\zeta}{\Gamma(\sigma)}+ \frac{(\varphi(T)-\varphi(a))^{1-\sigma+\varrho}}{\Gamma(\varrho+1)} \left(\delta +\mathcal{M}^* \right) \\
&:= \eta.           
\end{align*}
This gives  $\left\| \mathbb{F}u \right\|_{{\mathcal{PC}}_{{1-\sigma};\,\varphi}\left( J,\,\mathbb{R}\right)}\leq \eta.$ Thus the proof of the operator $\mathbb{F}$  maps bounded sets into bounded set is completed.

Step 3:  $\mathbb{F}$ is equicontinuous.\\
Let $u\in  C_{{1-\sigma};\,\varphi}\left( J,\,\mathbb{R}\right) $and $t_1 , t_2 \in J$ with $a< t_1 <t_2 < T,$ we get
\begin{align*}
&\left|(\mathbb{F}u)(t_2)-(\mathbb{F}u)(t_1) \right| \\
& = \left|\frac{(\varphi(t_2)-\varphi(a))^{\sigma-1}}{\Gamma(\sigma)}\, \left(u_a +\sum_{a<t_k<t_2}\mathcal{J}_k(u(t_k^-))\right)+\mathbf{I}_{a^+}^{\varrho; \, \varphi}f(t_2,u(t_2))\right.\\
&\qquad \left. - \frac{(\varphi(t_1)-\varphi(a))^{\sigma-1}}{\Gamma(\sigma)}\, \left(u_a +\sum_{a<t_k<t_1}\mathcal{J}_k(u(t_k^-))\right)- \mathbf{I}_{a^+}^{\varrho; \, \varphi}f(t_1,u(t_1)) \right| \\
& \leq |u_a|\, \left|\frac{(\varphi(t_2)-\varphi(a))^{\sigma-1}}{\Gamma(\sigma)}- \frac{(\varphi(t_1)-\varphi(a))^{\sigma-1}}{\Gamma(\sigma)}\right| \\
&\qquad + \left|\frac{(\varphi(t_2)-\varphi(a))^{\sigma-1}}{\Gamma(\sigma)}\sum_{a<t_k<t_2}\mathcal{J}_k(u(t_k^-))-\frac{(\varphi(t_1)-\varphi(a))^{\sigma-1}}{\Gamma(\sigma)}\sum_{a<t_k<t_1}\mathcal{J}_k(u(t_k^-)) \right| \\
&\qquad \qquad + \frac{1}{\Gamma(\varrho)}\int_{a}^{t_2}\varphi^{'}(s)(\varphi(t_2)-\varphi(s))^{\varrho-1}|f(s,u(s))|\, ds\\
&\qquad \qquad - \frac{1}{\Gamma(\varrho)}\int_{a}^{t_1}\varphi^{'}(s)(\varphi(t_1)-\varphi(s))^{\varrho-1}|f(s,u(s))|\, ds\\
& \leq |u_a|\, \left|\frac{(\varphi(t_2)-\varphi(a))^{\sigma-1}}{\Gamma(\sigma)}- \frac{(\varphi(t_1)-\varphi(a))^{\sigma-1}}{\Gamma(\sigma)}\right|+ \left|\frac{(\varphi(t_2)-\varphi(a))^{\sigma-1}}{\Gamma(\sigma)}\sum_{a<t_k< t_2 - t_1} \mathcal{J}_k(u(t_k^-) \right|\\
&\qquad + \frac{1}{\Gamma(\varrho)} \int_{a}^{t_2}\varphi^{'}(s)(\varphi(t_2)-\varphi(s))^{\varrho-1}\,(\varphi(s)-\varphi(a))^{\sigma-1}\,(\varphi(s)-\varphi(a))^{1-\sigma}|f(s,u(s))|\, ds\\
&\qquad- \frac{1}{\Gamma(\varrho)} \int_{a}^{t_1}\varphi^{'}(s)(\varphi(t_1)-\varphi(s))^{\varrho-1}\,(\varphi(s)-\varphi(a))^{\sigma-1}\,(\varphi(s)-\varphi(a))^{1-\sigma}|f(s,u(s))|\, ds\\  
& \leq |u_a|\, \left|\frac{(\varphi(t_2)-\varphi(a))^{\sigma-1}}{\Gamma(\sigma)}- \frac{(\varphi(t_1)-\varphi(a))^{\sigma-1}}{\Gamma(\sigma)}\right|+\frac{(\varphi(t_2)-\varphi(a))^{\sigma-1}}{\Gamma(\sigma)}\sum_{a<t_k< t_2 - t_1} \left|\mathcal{J}_k(u(t_k^-) \right|\\
&\qquad + \left\|f \right\|_{{\mathcal{PC}}_{{1-\sigma};\,\varphi}\left( J,\,\mathbb{R}\right)}\left(\mathbf{I}_{a^+}^{\varrho; \, \varphi}(\varphi(t_2)-\varphi(a))^{\sigma-1}- \mathbf{I}_{a^+}^{\varrho; \, \varphi}(\varphi(t_1)-\varphi(a))^{\sigma-1} \right)\\
&= \leq |u_a|\, \left|\frac{(\varphi(t_2)-\varphi(a))^{\sigma-1}}{\Gamma(\sigma)}- \frac{(\varphi(t_1)-\varphi(a))^{\sigma-1}}{\Gamma(\sigma)}\right|+\frac{(\varphi(t_2)-\varphi(a))^{\sigma-1}}{\Gamma(\sigma)}\sum_{a<t_k< t_2 - t_1} \left|\mathcal{J}_k(u(t_k^-) \right|\\
&\qquad + \left\|f \right\|_{{\mathcal{PC}}_{{1-\sigma};\,\varphi}\left( J,\,\mathbb{R}\right)} \left\lbrace\frac{\Gamma(\sigma)}{\Gamma(\varrho+\sigma)}\left[ (\varphi(t_2)-\varphi(a))^{\varrho+\sigma-1} - (\varphi(t_1)-\varphi(a))^{\varrho+\sigma-1}\right]  \right\rbrace .   
\end{align*}
Therefore
$$
\left|(\mathbb{F}u)(t_2)-(\mathbb{F}u)(t_1) \right|\to 0 ~\text{as}\,\, |t_1-t_2 | \to 0.
$$
This shows that $\mathbb{F}$ is equi-continuous on $J$. By utilizing ${\mathcal{PC}}_{1-\sigma;\, \varphi}$ type of Arzela-Ascoli Theorem \ref{pc}, the operator $\mathbb{F}$ is completely continuous. \\
Step 4: Finally, we show that the set 
$$
\mathcal{D}=\left\lbrace u\in {\mathcal{PC}}_{{1-\sigma};\,\varphi}\left( J,\,\mathbb{R}\right): u=\lambda \mathbb{F}u,\, \text{for some} \,\lambda \in (0,1)  \right\rbrace 
$$ is bounded.
Let any $u\in \mathcal{D}$. Then
\begin{align*}
|u(t)|=|\lambda\,\mathbb{F}u(t)| < |\mathbb{F}u(t)|.
\end{align*}
Let $\mathcal{N}^* = \max _{a<t_k<t}|\mathcal{J}_k(0)|.$ Then, using hypotheses $(H_1)(i)$ and $(H_2)(i)$, for each $t\in J,$ we have
\begin{align*}
&(\varphi(t)-\varphi(a))^{1-\sigma}|u(t)|\\
&< \frac{|u_a|}{\Gamma(\sigma)}\,+ \frac{1}{\Gamma(\sigma)} \sum_{a<t_k<t}\left| \mathcal{J}_k(u(t_k^-))\right|+ (\varphi(t)-\varphi(a))^{1-\sigma} \left| \mathbf{I}_{a^+}^{\varrho; \, \varphi}f(t,u(t))\right|\\
&\leq \frac{|u_a|}{\Gamma(\sigma)}+ \frac{1}{\Gamma(\sigma)}\sum_{a<t_k<t}\left| \mathcal{J}_k(u(t_k^-)-\mathcal{J}_k(o(t_k^-))\right| + \frac{1}{\Gamma(\sigma)}\sum_{a<t_k<t}\left|\mathcal{J}_k(o(t_k^-))\right| \\
& \qquad + (\varphi(t)-\varphi(a))^{1-\sigma} \left| \mathbf{I}_{a^+}^{\varrho; \, \varphi}\left( f(t,u(t))-f(t,0)\right) \right|\\
&\qquad\qquad +(\varphi(t)-\varphi(a))^{1-\sigma}\left| \mathbf{I}_{a^+}^{\varrho; \, \varphi}f(t,0) \right|\\
& \leq \frac{|u_a|}{\Gamma(\sigma)}+ \frac{1}{\Gamma(\sigma)}\sum_{a<t_k<t}(\varphi(t_k)-\varphi(a))^{1-\sigma}|u(t_k^-)| + \frac{1}{\Gamma(\sigma)}\sum_{a<t_k<t}|\mathcal{J}_k(0)|\\
&\qquad +\frac{(\varphi(t)-\varphi(a))^{1-\sigma}}{\Gamma(\varrho)} \int_{a}^{t}\varphi^{'}(s)(\varphi(t)-\varphi(s))^{\varrho-1}\,(\varphi(s)-\varphi(a))^{1-\sigma}\left|u(s) \right| \, ds\\
&\qquad + \frac{(\varphi(t)-\varphi(a))^{1-\sigma}}{\Gamma(\varrho)} \int_{a}^{t}\varphi^{'}(s)(\varphi(t)-\varphi(s))^{\varrho-1}\,\left|f(s,0) \right| \, ds\\
&\leq \left( \frac{|u_a|}{\Gamma(\sigma)}+\frac{m\, \mathcal{N}^*}{\Gamma(\sigma)}+ \frac{\mathcal{M}^*(\varphi(T)-\varphi(a))^{1-\sigma+\varrho}}{\Gamma(\varrho +1)} \right) \\
& \qquad + \frac{(\varphi(t)-\varphi(a))^{1-\sigma}}{\Gamma(\varrho)} \int_{a}^{t}\varphi^{'}(s)(\varphi(t)-\varphi(s))^{\varrho-1}\,(\varphi(s)-\varphi(a))^{1-\sigma}\left|u(s) \right| \, ds\\
&\qquad + \frac{1}{\Gamma(\sigma)}\sum_{a<t_k<t}(\varphi(t_k)-\varphi(a))^{1-\sigma}|u(t_k^-)|.
\end{align*}
Now applying the Lemma \ref{JI} with
\begin{align*}
\mathcal{U}(t)&=(\varphi(t)-\varphi(a))^{1-\sigma}\left|u(t) \right|,\\
\mathcal{V}(t)&= \frac{|u_a|}{\Gamma(\sigma)}+\frac{m\, \mathcal{N}^*}{\Gamma(\sigma)}+ \frac{\mathcal{M}^*(\varphi(T)-\varphi(a))^{1-\sigma+\varrho}}{\Gamma(\varrho +1)},\\
\mathbf{g}(t)&=\frac{(\varphi(t)-\varphi(a))^{1-\sigma}}{\Gamma(\varrho)},\\
\beta_k&=\frac{1}{\Gamma(\sigma)},
\end{align*}
we obtain
\begin{align*}
&(\varphi(t)-\varphi(a))^{1-\sigma}|u(t)| \\
&\leq \left( \frac{|u_a|}{\Gamma(\sigma)}+\frac{m\, \mathcal{N}^*}{\Gamma(\sigma)}+ \frac{\mathcal{M}^*(\varphi(T)-\varphi(a))^{1-\sigma+\varrho}}{\Gamma(\varrho +1)} \right)\\
&\qquad \times \left[\prod_{i=1}^{k} \left\lbrace 1+\frac{1}{\Gamma(\sigma)} E_\varrho\left( \frac{(\varphi(t)-\varphi(a))^{1-\sigma}}{\Gamma(\varrho)}\Gamma(\varrho)(\varphi(t_i)-\varphi(a))^\varrho\right) \right\rbrace  \right] \\
& \qquad \times E_\varrho\left( \frac{(\varphi(t)-\varphi(a))^{1-\sigma}}{\Gamma(\varrho)}\Gamma(\varrho)(\varphi(t)-\varphi(a))^\varrho\right).
\end{align*}
 Since $\varphi$ is increasing,  $(\varphi(t_i)-\varphi(a))^\varrho\leq(\varphi(t)-\varphi(a))^\varrho$, ~ $k=1,2,\cdots,m$. Therefore, we have 
 \begin{align*}
(\varphi(t)-\varphi(a))^{1-\sigma}\left|u(t) \right|
 &\leq \left( \frac{|u_a|}{\Gamma(\sigma)}+\frac{m\, \mathcal{N}^*}{\Gamma(\sigma)}+ \frac{\mathcal{M}^*(\varphi(T)-\varphi(a))^{1-\sigma+\varrho}}{\Gamma(\varrho +1)} \right)\\
 &\qquad \times \left(1+\frac{1}{\Gamma(\sigma)} E_\varrho ((\varphi(T)-\varphi(a))^{1-\sigma+\varrho}) \right)^m\\
 &\qquad \times E_\varrho\left((\varphi(T)-\varphi(a))^{1-\sigma+\varrho} \right) 
 \end{align*}
 Therefore for each $t\in J$, we have
  $$(\varphi(t)-\varphi(a))^{1-\sigma}\left|u(t) \right| \leq \left( \frac{|u_a|}{\Gamma(\sigma)}+\frac{m\, \mathcal{N}^*}{\Gamma(\sigma)}+ \frac{\mathcal{M}^*(\varphi(T)-\varphi(a))^{1-\sigma+\varrho}}{\Gamma(\varrho +1)} \right)\, \mathcal{A}_{m,\varrho},$$
  where
\begin{equation} \label{32}
\mathcal{A}_{m,\varrho}= \left(1+\frac{1}{\Gamma(\sigma)}E_\varrho ((\varphi(T)-\varphi(a))^{1-\sigma+\varrho}) \right)^m
    E_\varrho\left((\varphi(T)-\varphi(a))^{1-\sigma+\varrho} \right).
\end{equation}  
Thus
 $$
 \left\|u \right\|_{{\mathcal{PC}}_{{1-\sigma};\,\varphi}\left( J,\,\mathbb{R}\right)} \leq  \mathcal{A}_{m,\varrho}\,\left( \frac{|u_a|}{\Gamma(\sigma)}+\frac{m\, \mathcal{N}^*}{\Gamma(\sigma)}+ \frac{\mathcal{M}^*(\varphi(T)-\varphi(a))^{1-\sigma+\varrho}}{\Gamma(\varrho +1)} \right). 
 $$
 This shows that the set $\mathcal{D}$ is bounded subset of ${\mathcal{PC}}_{{1-\sigma};\,\varphi}\left( J,\,\mathbb{R}\right).$ 
 Hence by Schaefer's fixed point theorem $\mathbb{F}$ has at least one fixed point in  ${\mathcal{PC}}_{{1-\sigma};\,\varphi}\left( J,\,\mathbb{R}\right)$ which is the solution of problem \eqref{p1}.
\end{proof}

\section{Continuous dependence of solution}
This section deals with the continuous dependence of solutions of the problem \eqref{p1} on initial conditions and the functions involved in \eqref{p1} on the right hand sides.
\subsection{Continuous Dependence on initial conditions} 

\begin{theorem} Suppose that the functions $f:(a,T] \to \mathbb{R}$ and $\mathcal{J}_k:\mathbb{R} \to \mathbb{R}$ satisfy the hypotheses $(H_1)(ii)$ and $(H_2)(i)$ respectively. Let $u,v \in {\mathcal{PC}}_{{1-\sigma};\,\varphi}\left( J,\,\mathbb{R}\right) $ are the solutions of the problem 
\begin{align} 
\begin{cases}
 ^H \mathbf{D}^{\varrho,\, \nu; \, \varphi}_{a^+}u(t)=f(t, u(t)),~t \in J-\{t_1, t_2,\cdots ,t_m\},\\
\Delta \mathbf{I}_{a^+}^{1-\sigma; \, \varphi}u(t_k)= \mathcal{J}_k(u(t_k^-)),  ~k = 1,2,\cdots,m, \label{p2} 
\end{cases}
\end{align} 
corresponding to ~$\mathbf{I}_{a^+}^{1-\sigma; \, \varphi}u(a)= u_a$
and ~$\mathbf{I}_{a^+}^{1-\sigma; \, \varphi}u(a)= v_a $ respectively. Then,
\begin{align} \label{53}
\left\|u-v \right\|_{{\mathcal{PC}}_{{1-\sigma};\,\varphi}\left( J,\,\mathbb{R}\right)} &\leq \frac{|u_a- v_a|}{\Gamma(\sigma)} \left(1+\frac{1}{\Gamma(\sigma)} E_\varrho (\varphi(T)-\varphi(a))^{1-\sigma+\varrho} \right)^m
  \nonumber\\
    & \qquad \times E_\varrho\left((\varphi(T)-\varphi(a))^{1-\sigma+\varrho} \right).  
\end{align}

\end{theorem}
\begin{proof}
Let $u,v \in {\mathcal{PC}}_{{1-\sigma};\,\varphi}\left( J,\,\mathbb{R}\right) $ are the solutions of the problem \eqref{p2} corresponding to ~$\mathbf{I}_{a^+}^{1-\sigma; \, \varphi}v(a)= u_a$
and ~$\mathbf{I}_{a^+}^{1-\sigma; \, \varphi}u(a)= v_a $ respectively. Then
in the view of lemma \ref{fie}, we have
$$
u(t)= \frac{(\varphi(t)-\varphi(a))^{\sigma-1}}{\Gamma(\sigma)}\, \left(u_a +\sum_{a<t_k<t}\mathcal{J}_k(u(t_k^-))\right)+\mathbf{I}_{a^+}^{\varrho; \, \varphi}f(t,u(t))
$$ 
and
$$
v(t)=  \frac{(\varphi(t)-\varphi(a))^{\sigma-1}}{\Gamma(\sigma)}\, \left(v_a +\sum_{a<t_k<t}\mathcal{J}_k(v(t_k^-))\right)+\mathbf{I}_{a^+}^{\varrho; \, \varphi}f(t,v(t)).
$$
Using the hypotheses $(H_1)(ii)$, $(H_2)(i)$ for any $t \in J$, we have
\begin{align*}
&(\varphi(t)-\varphi(a))^{1-\sigma}|u(t)-v(t)|\\
& \leq \frac{|u_a- v_a|}{\Gamma(\sigma)} + \frac{1}{\Gamma(\sigma)}\sum_{a<t_k<t}(\varphi(t_k)-\varphi(a))^{1-\sigma}\left|u(t_k^-)-v(t_k^-) \right|\\
&\qquad +\frac{(\varphi(t)-\varphi(a))^{1-\sigma}}{\Gamma(\varrho)}\int_{a}^{t}\varphi^{'}(s)(\varphi(t)-\varphi(s))^{\varrho-1}\,(\varphi(s)-\varphi(a))^{1-\sigma}|u(s)-v(s)|\, ds.
\end{align*}
By application of Lemma \eqref{JI} to the above inequality with $\mathcal{U}(t)$,
$\mathbf{g}(t)$ and $\beta_k$ as given in the proof of Theorem \ref{TUH} and 
$
\mathcal{V}(t)= \frac{|u_a- v_a|}{\Gamma(\sigma)},
$
we obtain 
\begin{align*}
&(\varphi(t)-\varphi(a))^{1-\sigma}|u(t)-v(t)|\\
& \leq \frac{|u_a- v_a|}{\Gamma(\sigma)}\,\left(1+\frac{1}{\Gamma(\sigma)} E_\varrho (\varphi(T)-\varphi(a))^{1-\sigma+\varrho} \right)^m
    E_\varrho\left((\varphi(T)-\varphi(a))^{1-\sigma+\varrho} \right),t \in J.
\end{align*} 
This gives the inequality \eqref{53}.
\end{proof}
\begin{rem} 
The inequality \eqref{53} gives dependence  of solution  of  impulsive $\varphi$--HFDE \eqref{p1} on initial conditions. Further, for $u_a=v_a$ the inequality \eqref{53} gives uniqueness of the solution also.
\end{rem}

\subsection{Continuous dependence on the functions} 
Now consider the following impulsive $\varphi$- HFDE
\begin{align}\label{p3} 
\begin{cases}
 ^H \mathbf{D}^{\varrho,\, \nu; \, \varphi}_{a^+}v(t)=\tilde{f}(t, v(t)),~t \in J -\{t_1, t_2,\cdots ,t_m\},\\
\Delta \mathbf{I}_{a^+}^{1-\sigma; \, \varphi}v(t_k)= \tilde{\mathcal{J}}_k(v(t_k^-)),  ~k = 1,2,\cdots,m, \\
 \mathbf{I}_{a^+}^{1-\sigma; \, \varphi}v(a)= v_a \in \mathbb{R},
\end{cases}
\end{align}
where $\tilde{f}:(a,T]\to \mathbb{R}$ and $\tilde{\mathcal{J}}_k: \mathbb{R}\to \mathbb{R}$ are the functions other than $f$ and $\mathcal{J}_k$ specified in the problem \eqref{p1}.
\begin{theorem}
Suppose that the functions $f:(a,T] \to \mathbb{R}$ and $\mathcal{J}_k:\mathbb{R} \to \mathbb{R}$ satisfy the hypotheses $(H_1)(ii)$ and $(H_2)(i)$ respectively.  Further, suppose that there is a constant $\delta_a >0, ~\varepsilon_f >0,~ \varepsilon_\mathcal{J}>0$ such that
\begin{align*}
|u_a-v_a|&\leq \delta_a,\\ 
|f(t,u(t))-\tilde{f}(t,u(t))|&\leq \varepsilon_f ,\\
|\mathcal{J}_k(u(t_k^-))-\tilde{\mathcal{J}}_k(u(t_k^-))|&\leq \varepsilon_\mathcal{J}.
\end{align*}
Then, the solution $u\in {\mathcal{PC}}_{{1-\sigma};\,\varphi}\left( J,\,\mathbb{R}\right) $ of  \eqref{p1} and the solution $v\in {\mathcal{PC}}_{{1-\sigma};\,\varphi}\left( J,\,\mathbb{R}\right) $ of \eqref{p3} satisfy the inequality
\begin{align} \label{63}
&\left\|u-v \right\|_{{\mathcal{PC}}_{{1-\sigma};\,\varphi}\left( J,\,\mathbb{R}\right)}\\
&\leq\left(\frac{1}{\Gamma(\sigma)}\delta_a\,
+\frac{m}{\Gamma(\sigma)} \, \,\varepsilon_\mathcal{J}\,+ \frac{(\varphi(T)-\varphi(a))^{1-\sigma+\varrho}}{\Gamma(\varrho +1)}\varepsilon_f \right) \nonumber \\ 
 &\qquad \times \left(1+\frac{1}{\Gamma(\sigma)} E_\varrho (\varphi(T)-\varphi(a))^{1-\sigma+\varrho} \right)^m
    E_\varrho\left((\varphi(T)-\varphi(a))^{1-\sigma+\varrho} \right). 
\end{align}
\end{theorem}
\begin{proof}
Let $u\in {\mathcal{PC}}_{{1-\sigma};\,\varphi}\left( J,\,\mathbb{R}\right) $ be the  solution of  \eqref{p1} and $v\in {\mathcal{PC}}_{{1-\sigma};\,\varphi}\left( J,\,\mathbb{R}\right) $ be the  solution of \eqref{p3}. Then, their corresponding equivalent integral equations are given by respectively
$$
u(t)= \frac{(\varphi(t)-\varphi(a))^{\sigma-1}}{\Gamma(\sigma)}\, \left(u_a +\sum_{a<t_k<t}\mathcal{J}_k(u(t_k^-))\right)+\mathbf{I}_{a^+}^{\varrho; \, \varphi}f(t,u(t))
$$ 
and
$$
v(t)=  \frac{(\varphi(t)-\varphi(a))^{\sigma-1}}{\Gamma(\sigma)}\, \left(v_a +\sum_{a<t_k<t}\tilde{\mathcal{J}_k}(v(t_k^-))\right)+\mathbf{I}_{a^+}^{\varrho; \, \varphi}\tilde{f}(t,v(t)).
$$
Using the hypotheses $(H_1)(ii)$ and $(H_2)(i)$, for any $t\in J$,  we have
\begin{align*}
|u(t)-v(t)|& \leq \frac{(\varphi(t)-\varphi(a))^{\sigma-1}}{\Gamma(\sigma)}\,\delta_a + \frac{(\varphi(t)-\varphi(a))^{\sigma-1}}{\Gamma(\sigma)}\sum_{a<t_k<t}\left|\mathcal{J}_k(u(t_k^-))-\mathcal{J}_k(v(t_k^-)) \right|\\
&\qquad + \frac{(\varphi(t)-\varphi(a))^{\sigma-1}}{\Gamma(\sigma)}\sum_{a<t_k<t}\left|\mathcal{J}_k(v(t_k^-))-\tilde{\mathcal{J}}_k(v(t_k^-)) \right|\\ 
&\qquad + \left|\mathbf{I}_{a^+}^{\varrho; \, \varphi}\left( f(t,u(t))-f(t,v(t))\right)  \right|+ \left|\mathbf{I}_{a^+}^{\varrho; \, \varphi}\left( f(t,v(t))-\tilde{f}(t,v(t))\right)  \right| \\
&\leq  \left( \frac{(\varphi(t)-\varphi(a))^{\sigma-1}}{\Gamma(\sigma)}\,\delta_a+\frac{(\varphi(t)-\varphi(a))^{\sigma-1}}{\Gamma(\sigma)} \, m\,\varepsilon_\mathcal{J} + \frac{(\varphi(t)-\varphi(a))^{\varrho}}{\Gamma(\varrho +1)}\varepsilon_f\right) \\
&\qquad + \frac{(\varphi(t)-\varphi(a))^{\sigma-1}}{\Gamma(\sigma)}\, \sum_{a<t_k<t}(\varphi(t_k^-)-\varphi(a))^{1-\sigma}\left|u(t_k^-)-v(t_k^-)\right| \\
&\qquad + \frac{1}{\Gamma(\varrho)}\int_{a}^{t}\varphi^{'}(s)(\varphi(t)-\varphi(s))^{\varrho-1}\,(\varphi(s)-\varphi(a))^{1-\sigma}|u(s)-v(s)|\, ds.
\end{align*}
Thus
\begin{align*}
&(\varphi(t)-\varphi(a))^{1-\sigma}|u(s)-v(s)|\\
&\leq \left(\frac{1}{\Gamma(\sigma)}\,\delta_a+\frac{m}{\Gamma(\sigma)} \, \,\varepsilon_\mathcal{J}\,+ \frac{(\varphi(t)-\varphi(a))^{1-\sigma+\varrho}}{\Gamma(\varrho +1)}\varepsilon_f \right)\\
&\qquad + \frac{1}{\Gamma(\sigma)}\, \sum_{a<t_k<t}(\varphi(t_k^-)-\varphi(a))^{1-\sigma}\left|u(t_k^-)-v(t_k^-)\right| \\
&\qquad + \frac{(\varphi(t)-\varphi(a))^{1-\sigma}}{\Gamma(\varrho)}\int_{a}^{t}\varphi^{'}(s)(\varphi(t)-\varphi(s))^{\varrho-1}\,(\varphi(s)-\varphi(a))^{1-\sigma}|u(s)-v(s)|\, ds, ~t \in J
\end{align*}
Applying lemma \ref{JI} to the above inequality with $\mathcal{U}(t)$,
$\mathbf{g}(t)$ and $\beta_k$ as given in the proof of Theorem \ref{TUH} and 
$$
\mathcal{V}(t)= \left(\frac{1}{\Gamma(\sigma)}\,\delta_a+\frac{m}{\Gamma(\sigma)} \, \,\varepsilon_\mathcal{J}\,+ \frac{(\varphi(t)-\varphi(a))^{1-\sigma+\varrho}}{\Gamma(\varrho +1)}\varepsilon_f \right),
$$
we obtain
\begin{align*}
&(\varphi(t)-\varphi(a))^{1-\sigma}|u(s)-v(s)|\\
&\leq \left(\frac{\delta_a}{\Gamma(\sigma)}\,+\frac{1}{\Gamma(\sigma)} \, m\,\varepsilon_\mathcal{J}\,+ \frac{(\varphi(T)-\varphi(a))^{1-\sigma+\varrho}}{\Gamma(\varrho +1)}\varepsilon_f \right)\\
&\qquad \times \left(1+\frac{1}{\Gamma(\sigma)} E_\varrho (\varphi(T)-\varphi(a))^{1-\sigma+\varrho} \right)^m
    E_\varrho\left((\varphi(T)-\varphi(a))^{1-\sigma+\varrho} \right), ~t \in J.
\end{align*}
This gives the inequality \eqref{63}.
\end{proof}
\begin{rem} $ $
\begin{itemize}
\item [(1)]Inequality \eqref{63} shows that the solution $u$ of \eqref{p1} depends continuously on the functions involved in right hand side of \eqref{p1}.
\item [(2)] If $\delta_a \to 0$, $\varepsilon_\mathcal{J} \to 0$, $\varepsilon_f \to 0 $, then we get the uniqueness of the solution of the problem \eqref{p1}.
\item [(3)] If $\varepsilon_f=\varepsilon_\mathcal{J} = 0$, then the inequality \eqref{63} gives the dependence of solution of the problem \eqref{p1} on the initial condition.
\end{itemize}
\end{rem}
\subsection{Continuous dependence on the order of $\varphi$-Hilfer derivative} 
Now consider the following impulsive $\varphi$- HFDE
\begin{align}\label{p4} 
\begin{cases}
 ^H \mathbf{D}^{\varrho-\delta ,\, \nu; \, \varphi}_{a^+}v(t)=f(t, v(t)),~t \in J -\{t_1, t_2,\cdots ,t_m\},\\
\Delta \mathbf{I}_{a^+}^{1-\sigma^{*}; \, \varphi}v(t_k)= {\mathcal{J}}_k(v(t_k^-)),  ~k = 1,2,\cdots,m, \\
 \mathbf{I}_{a^+}^{1-\sigma^{*}; \, \varphi}v(a)= v_a \in \mathbb{R},
\end{cases}
\end{align}
where $f$ and $\mathcal{J}_k$ specified in the problem \eqref{p1} and $\sigma^{*}=\sigma+\delta(\nu-1)$, $\delta>0$.
\begin{theorem}
Suppose that the functions $f:(a,T] \to \mathbb{R}$ and $\mathcal{J}_k:\mathbb{R} \to \mathbb{R}$ satisfy the hypotheses $(H_1)$ and $(H_2)$ respectively.  
Then, the solution $u\in {\mathcal{PC}}_{{1-\sigma};\,\varphi}\left( J,\,\mathbb{R}\right) $ of  \eqref{p1} and the solution $v\in {\mathcal{PC}}_{{1-\sigma^{*}};\,\varphi}\left( J,\,\mathbb{R}\right) $ of \eqref{p4} satisfy the inequality
\begin{align} \label{64}
&(\varphi(t)-\varphi(a))^{1-\sigma}|u(t)-v(t)| \nonumber\\
&\leq \mathcal{B}(t)\,\left(1+\frac{1}{\Gamma(\sigma)} E_\varrho (\varphi(T)-\varphi(a))^{1-\sigma+\varrho} \right)^m
    E_\varrho\left((\varphi(T)-\varphi(a))^{1-\sigma+\varrho} \right), 
\end{align}
where
\begin{align}\label{b}
\mathfrak{B}(t)=& \left| \frac{u_a}{\Gamma(\sigma)}-\frac{v_a (\varphi(t)-\varphi(a))^{\delta(\beta-1))}}{\Gamma(\sigma+\delta(\beta-1))}\right|+ \zeta\,\left| \frac{ 1}{\Gamma(\sigma)}- \frac{ (\varphi(t)-\varphi(a))^{(\delta(\beta-1))}}{\Gamma(\sigma+\delta(\beta-1))}\right| \nonumber\\
&\qquad+ \|f\|_{{\mathcal{PC}}_{{1-\sigma};\,\varphi}} \left\lbrace \frac{\Gamma(\sigma)(\varphi(t)-\varphi(a))^{\varrho}}{\Gamma(\sigma+\varrho)}- \frac{\Gamma(\sigma)(\varphi(t)-\varphi(a))^{\varrho-\delta}}{\Gamma(\sigma+\varrho-\delta)} \right\rbrace,
\end{align}
and
$
\zeta=\sum_{k=1}^{m}\zeta_k.
$
\end{theorem}
\begin{proof}
Let $u\in {\mathcal{PC}}_{{1-\sigma};\,\varphi}\left( J,\,\mathbb{R}\right) $ be the  solution of  \eqref{p1} and $v\in {\mathcal{PC}}_{{1-\sigma^{*}};\,\varphi}\left( J,\,\mathbb{R}\right) $ be the  solution of \eqref{p4}. Then, their corresponding equivalent integral equations are given by respectively
$$
u(t)= \frac{(\varphi(t)-\varphi(a))^{\sigma-1}}{\Gamma(\sigma)}\, \left(u_a +\sum_{a<t_k<t}\mathcal{J}_k(u(t_k^-))\right)+\mathbf{I}_{a^+}^{\varrho; \, \varphi}f(t,u(t))
$$ 
and
$$
v(t)=  \frac{(\varphi(t)-\varphi(a))^{\sigma^{'}-1}}{\Gamma(\sigma^{'})}\, \left(v_a +\sum_{a<t_k<t}{\mathcal{J}_k}(v(t_k^-))\right)+\mathbf{I}_{a^+}^{\varrho-\delta ; \, \varphi}f(t,v(t)).
$$
Using the hypotheses $(H_1)$ and $(H_2)$, for any $t\in J$,  we have
\begin{align*}
|u(t)-v(t)|
& \leq \left| \frac{u_a (\varphi(t)-\varphi(a))^{\sigma-1}}{\Gamma(\sigma)}-\frac{v_a (\varphi(t)-\varphi(a))^{(\sigma^{'}-1)}}{\Gamma(\sigma^{'})}\right| \\
&\qquad + \left|\frac{(\varphi(t)-\varphi(a))^{\sigma-1}}{\Gamma(\sigma)}\sum_{a<t_k<t}\mathcal{J}_k(u(t_k^-))- \frac{(\varphi(t)-\varphi(a))^{(\sigma^{'}-1)}}{\Gamma(\sigma^{'})}\sum_{a<t_k<t}\mathcal{J}_k(v(t_k^-))\right| \\
&\qquad \qquad+ \left|\mathbf{I}_{a^+}^{\varrho; \, \varphi}f(t,u(t))- \mathbf{I}_{a^+}^{\varrho-\delta ; \, \varphi}f(t,v(t)) \right|\\
&\leq  \left| \frac{u_a (\varphi(t)-\varphi(a))^{\sigma-1}}{\Gamma(\sigma)}-\frac{v_a (\varphi(t)-\varphi(a))^{(\sigma-1+\delta(\beta-1))}}{\Gamma(\sigma+\delta(\beta-1))}\right|\\
&\qquad+ \frac{ (\varphi(t)-\varphi(a))^{\sigma-1}}{\Gamma(\sigma)} \sum_{a<t_k<t}\left|\mathcal{J}_k(u(t_k^-))-\mathcal{J}_k(v(t_k^-))\right|\\
&\qquad+ \left| \frac{ (\varphi(t)-\varphi(a))^{\sigma-1}}{\Gamma(\sigma)}- \frac{ (\varphi(t)-\varphi(a))^{(\sigma-1+\delta(\beta-1))}}{\Gamma(\sigma-1+\delta(\beta-1))}\right|  \sum_{a<t_k<t}\left| \mathcal{J}_k(v(t_k^-)) \right|\\
&\qquad +\frac{1}{\Gamma(\varrho)} \int_{a}^{t}\varphi^{'}(s)(\varphi(t)-\varphi(s))^{\varrho-1}\,\left| f(s, u(s))-f(s, v(s))\right| \, ds\\
&\qquad + \int_{a}^{t}\varphi^{'}(s)\left(\frac{(\varphi(t)-\varphi(s))^{\varrho-1}}{\Gamma(\varrho)}-\frac{(\varphi(t)-\varphi(s))^{\varrho-\delta-1}}{\Gamma(\varrho-\delta)} \right) \,\left| f(s, v(s))\right| \, ds\\
&\leq \left| \frac{u_a (\varphi(t)-\varphi(a))^{\sigma-1}}{\Gamma(\sigma)}-\frac{v_a (\varphi(t)-\varphi(a))^{(\sigma-1+\delta(\beta-1))}}{\Gamma(\sigma+\delta(\beta-1))}\right|\\
&\qquad + \frac{ (\varphi(t)-\varphi(a))^{\sigma-1}}{\Gamma(\sigma)} \sum_{a<t_k<t}(\varphi(t_k)-\varphi(a))^{\sigma-1}\left|u(t_k^-)-v(t_k^-)\right|\\
&\qquad + \zeta\,\left| \frac{ (\varphi(t)-\varphi(a))^{\sigma-1}}{\Gamma(\sigma)}- \frac{ (\varphi(t)-\varphi(a))^{(\sigma-1+\delta(\beta-1))}}{\Gamma(\sigma-1+\delta(\beta-1))}\right|\\
&\qquad + \frac{1}{\Gamma(\varrho)} \int_{a}^{t}\varphi^{'}(s)(\varphi(t)-\varphi(s))^{\varrho-1}\,(\varphi(s)-\varphi(a))^{1-\sigma}\,\left|u(s)- v(s)\right| \, ds\\
&\qquad + \int_{a}^{t}\varphi^{'}(s)\left(\frac{(\varphi(t)-\varphi(s))^{\varrho-1}}{\Gamma(\varrho)}-\frac{(\varphi(t)-\varphi(s))^{\varrho-\delta-1}}{\Gamma(\varrho-\delta)} \right)\\
&\qquad\qquad\qquad \,(\varphi(s)-\varphi(a))^{\sigma-1} (\varphi(s)-\varphi(a))^{1-\sigma} \left| f(s, v(s))\right| \, ds\\
&\leq \left| \frac{u_a (\varphi(t)-\varphi(a))^{\sigma-1}}{\Gamma(\sigma)}-\frac{v_a (\varphi(t)-\varphi(a))^{(\sigma-1+\delta(\beta-1))}}{\Gamma(\sigma+\delta(\beta-1))}\right|\\
&\qquad+  \frac{ (\varphi(t)-\varphi(a))^{\sigma-1}}{\Gamma(\sigma)} \sum_{a<t_k<t}(\varphi(t_k)-\varphi(a))^{\sigma-1}\left|u(t_k^-)-v(t_k^-)\right|\\
&\qquad + \zeta\,\left| \frac{ (\varphi(t)-\varphi(a))^{\sigma-1}}{\Gamma(\sigma)}- \frac{ (\varphi(t)-\varphi(a))^{(\sigma-1+\delta(\beta-1))}}{\Gamma(\sigma-1+\delta(\beta-1))}\right|\\
&\qquad + \frac{1}{\Gamma(\varrho)} \int_{a}^{t}\varphi^{'}(s)(\varphi(t)-\varphi(s))^{\varrho-1}\,(\varphi(s)-\varphi(a))^{1-\sigma}\,\left|u(s)- v(s)\right| \, ds\\
&\qquad + \|f\|_{{\mathcal{PC}}_{{1-\sigma};\,\varphi}} \left\lbrace \frac{\Gamma(\sigma)(\varphi(t)-\varphi(a))^{\varrho+\sigma-1}}{\Gamma(\sigma+\varrho)}- \frac{\Gamma(\sigma)(\varphi(t)-\varphi(a))^{\varrho-\delta+\sigma-1}}{\Gamma(\sigma+\varrho-\delta)} \right\rbrace.
\end{align*}
Therefore, for each $t\in J$ we have
\begin{align}\label{e1111}
&(\varphi(t)-\varphi(a))^{1-\sigma}|u(t)-v(t)| \nonumber\\
&\leq \mathfrak{B}(t)+ \frac{ 1}{\Gamma(\sigma)} \sum_{a<t_k<t}(\varphi(t_k)-\varphi(a))^{1-\sigma}\left|u(t_k^-)-v(t_k^-)\right|\nonumber\\
&\qquad +  \frac{(\varphi(t)-\varphi(a))^{1-\varrho}}{\Gamma(\varrho)} \int_{a}^{t}\varphi^{'}(s)(\varphi(t)-\varphi(s))^{\varrho-1}\,(\varphi(s)-\varphi(a))^{1-\sigma}\,\left|u(s)- v(s)\right| \, ds,
\end{align}
where $\mathfrak{B}(t)$ is as defined in \eqref{b}.
Now applying the Lemma \ref{JI} to \eqref{e1111} with
\begin{align*}
\mathcal{U}(t)=(\varphi(t)-\varphi(a))^{1-\sigma}|u(t)-v(t)|,\, \mathcal{V}(t)= \mathfrak{B}(t),\,
\mathbf{g}(t)=\frac{(\varphi(t)-\varphi(a))^{1-\sigma}}{\Gamma(\varrho)},\,\text{and} \,  
\beta_k =\frac{1}{\Gamma(\sigma)},
\end{align*}
we obtain
\begin{align*}
&(\varphi(t)-\varphi(a))^{1-\sigma}|u(t)-v(t)| \\
&\leq \mathcal{B}(t)\,\left[\prod_{i=1}^{k} \left\lbrace 1+\frac{1}{\Gamma(\sigma)} E_\varrho\left( \frac{(\varphi(t)-\varphi(a))^{1-\sigma}}{\Gamma(\varrho)}\Gamma(\varrho)(\varphi(t_i)-\varphi(a))^\varrho\right) \right\rbrace  \right] \\
& \qquad \times E_\varrho\left( \frac{(\varphi(t)-\varphi(a))^{1-\sigma}}{\Gamma(\varrho)}\Gamma(\varrho)(\varphi(t)-\varphi(a))^\varrho\right)\\
&=\mathcal{B}(t)\,\left[\prod_{i=1}^{k} \left\lbrace 1+\frac{1}{\Gamma(\sigma)} E_\varrho\left( (\varphi(t_i)-\varphi(a))^{1-\sigma+\varrho} \right) \right\rbrace  \right] \, E_\varrho\left( (\varphi(t)-\varphi(a))^{1-\sigma+\varrho} \right)\\
& \leq \mathcal{B}(t)\,\left(1+\frac{1}{\Gamma(\sigma)} E_\varrho (\varphi(T)-\varphi(a))^{1-\sigma+\varrho} \right)^m
    E_\varrho\left((\varphi(T)-\varphi(a))^{1-\sigma+\varrho} \right).
\end{align*}
\end{proof}
\begin{rem} The inequality \eqref{64} gives not only the dependence of solution on the order of $\varphi$-Hilfer derivative but also gives the dependency of solution on the initial condition. Indeed, 
\begin{itemize}

\item [(1)] if $u_a \neq v_a $ and taking $\delta \to 0$ then the inequality \eqref{64} gives dependency of solution on initial condition. 
\item [(2)] if $u_a = v_a $  then the inequality \eqref{64} gives dependency of solution on the order of $\varphi$-Hilfer derivative.
\end{itemize}
\end{rem}
\section{Ulam Stabilities of $\varphi$--HFDE}
In this section, we investigate the  Ulam--Hyers stabilities of the impulsive $\varphi$-HFDE \eqref{p1}.

Let $\tilde{u}\in {\mathcal{PC}}_{{1-\sigma};\,\varphi}\left( J,\,\mathbb{R}\right)$, ~$\chi >0, ~\epsilon>0$ and $\theta : J\to \mathbb{R}$ be a nondecreasing function. We consider the following inequalities:
\begin{align}
\begin{cases}
|^H \mathbf{D}^{\varrho,\nu;\,\varphi}_{a^+}  \tilde{u}(t)-f(t,\tilde{u}(t))|\leq \epsilon,~ t \in J \\
 |\Delta \mathbf{I}_{a^+}^{1-\sigma; \, \varphi}\tilde{u}(t_k)- \mathcal{J}_k(\tilde{u}(t_k^-))| \leq \epsilon, ~k=1,2,\cdots,m, \label{HU1}
\end{cases}
\end{align}
\begin{align}
\begin{cases}
|^H\mathbf{D}^{\varrho,\nu;\,\varphi}_{a^+}\tilde{u}(t)-f(t,\tilde{u}(t))|\leq \theta(t),~t \in J \\
 |\Delta \mathbf{I}_{a^+}^{1-\sigma; \, \varphi}\tilde{u}(t_k)- \mathcal{J}_k(\tilde{u}(t_k^-))| \leq \chi, ~k=1,2,\cdots,m \label{HU2}
\end{cases}
\end{align}
and
\begin{align}
\begin{cases}
|^H\mathbf{D}^{\varrho,\nu;\,\varphi}_{a^+}\tilde{u}(t)-f(t,\tilde{u}(t))|\leq \epsilon \, \theta(t),~ t \in J\\
|\Delta \mathbf{I}_{a^+}^{1-\sigma; \, \varphi}\tilde{u}(t_k)- \mathcal{J}_k(\tilde{u}(t_k^-))| \leq \epsilon\, \chi, ~k=1,2,\cdots,m. \label{HU3}
\end{cases}
\end{align}
 \begin{definition}
The problem \eqref{p1} is said to be Ulam--Hyers (UH) stable if for $\epsilon>0$ there exists a constant $C_{m,\varrho}>0$ such that, for every solution $\tilde{u}\in {\mathcal{PC}}_{{1-\sigma};\,\varphi}\left( J,\,\mathbb{R}\right) $ of the inequality \eqref{HU1}, there is a unique solution $u\in {\mathcal{PC}}_{{1-\sigma};\,\varphi}\left( J,\,\mathbb{R}\right)$ to the problem \eqref{p1} satisfying 
$$
\left\|\tilde{u}-u \right\|_{{\mathcal{PC}}_{{1-\sigma};\,\varphi}\left( J,\,\mathbb{R}\right)}\leq C_{m,\varrho}\,\epsilon, ~t\in J. 
$$
 \end{definition}
 \begin{definition}
The problem \eqref{p1} is said to be generalized Ulam--Hyers (GUH) stable if 
there exists  ~$\phi_{m,\varrho}\in C(\mathbb{R^+},\mathbb{R^+})$ with $\phi_{m,\varrho}(0)=0$ such that for each solution $\tilde{u}\in {\mathcal{PC}}_{{1-\sigma};\,\varphi}\left( J,\,\mathbb{R}\right) $ of the inequality \eqref{HU2} there is a unique solution $u\in {\mathcal{PC}}_{{1-\sigma};\,\varphi}\left( J,\,\mathbb{R}\right)$ to the problem \eqref{p1} satisfying 
$$
\left\|\tilde{u}-u \right\|_{{\mathcal{PC}}_{{1-\sigma};\,\varphi}\left( J,\,\mathbb{R}\right)}\leq \,\phi_{m,\varrho}(\epsilon), ~t\in J. 
$$  
 \end{definition}
 \begin{definition}
The problem \eqref{p1} is said to be Ulam--Hyers--Rassias (UHR) stable corresponding to $(\theta,\chi)$ if for every $\epsilon>0$ there exists a real number $C_{m,\varrho,\theta}>0$ such that, for every solution $\tilde{u}\in {\mathcal{PC}}_{{1-\sigma};\,\varphi}\left( J,\,\mathbb{R}\right) $ of the inequality \eqref{HU3}, there is a unique solution $u\in {\mathcal{PC}}_{{1-\sigma};\,\varphi}\left( J,\,\mathbb{R}\right)$ to the problem \eqref{p1} satisfying 
$$
(\varphi(t)-\varphi(a))^{1-\sigma}\,\left|\tilde{u}(t)-u(t) \right|\leq C_{m,\varrho,\theta}\,\epsilon\,(\theta(t)+\chi), ~t\in J. 
$$
 \end{definition}
 \begin{definition}
The problem \eqref{p1} is said to be generalized Ulam--Hyers--Rassias (GUHR)stable corresponding to $(\theta,\chi)$ if there exists a constant $C_{m,\varrho,\theta}>0$ such that, for every solution $\tilde{u}\in {\mathcal{PC}}_{{1-\sigma};\,\varphi}\left( J,\,\mathbb{R}\right) $ of the inequality \eqref{HU2},there is a unique solution $u\in {\mathcal{PC}}_{{1-\sigma};\,\varphi}\left( J,\,\mathbb{R}\right)$ to the problem \eqref{p1} satisfying 
$$
(\varphi(t)-\varphi(a))^{1-\sigma}\,\left|\tilde{u}(t)-u(t) \right|\leq C_{m,\varrho,\theta}\,(\theta(t)+\chi), ~~~t\in J. 
$$
 \end{definition}
 \begin{rem}\label{rm}
 The function $\tilde{u}\in {\mathcal{PC}}_{{1-\sigma};\,\varphi}\left( J,\,\mathbb{R}\right) $ is called a solution of the inequality \eqref{HU1} if there is a function $\mathcal{E} \in {\mathcal{PC}}_{{1-\sigma};\,\varphi}\left( J,\,\mathbb{R}\right)$ together with a sequence $\{\mathcal{E}_k\},\,k=1,2,\cdots,m$ depending on $\tilde{u}$ with
 \begin{itemize}
\item[\rm (1)]$|\mathcal{E}(t)|\leq \epsilon,\,|\mathcal{E}_k|\leq \epsilon,\, t\in J,\, k=1,2,\cdots,m,$
\item[\rm (2)] $^H \mathbf{D}^{\varrho,\, \nu; \, \varphi}_{a^+}\tilde{u}(t)=f(t, \tilde{u}(t))+\mathcal{E}(t),\,t\in J,$
\item[\rm (3)]$\Delta \mathbf{I}_{a^+}^{1-\sigma; \, \varphi}\tilde{u}(t_k)= \mathcal{J}_k(u(t_k^-))+\mathcal{E}_k,  ~k = 1,2,\cdots,m.$
 \end{itemize}
 \end{rem}
 
Looking towards the Remark \ref{rm}, one can state similar types of remark for the inequalities \eqref{HU2} and \eqref{HU3}.
\begin{theorem}\label{TUH}
If the hypotheses $(H_1)(ii)$ and $(H_2)(i)$ are satisfied then the problem \eqref{p1} is UH stable.
\end{theorem}
\begin{proof}
Let  $\tilde{u}\in {\mathcal{PC}}_{{1-\sigma};\,\varphi}\left( J,\,\mathbb{R}\right) $ be the  solution of the inequality \eqref{HU1}.  Then in the view of  Remark \ref{rm}, we have
\begin{align}\label{55}
\tilde{u}(t)&=\frac{(\varphi(t)-\varphi(a))^{\sigma-1}}{\Gamma(\sigma)}\, \left(\mathbf{I}_{a^+}^{1-\sigma; \, \varphi}\tilde{u}(a) +\sum_{a<t_k<t}\mathcal{J}_k(\tilde{u}(t_k^-))+\sum_{a<t_k<t} \mathcal{E}_k \right) \nonumber\\
&\qquad+\mathbf{I}_{a^+}^{\varrho; \, \varphi}f(t,\tilde{u}(t))+\mathbf{I}_{a^+}^{\varrho; \, \varphi}\mathcal{E}(t).
\end{align}

Therefore, we have
\begin{align} \label{56}
&\left|\tilde{u}(t)-\frac{(\varphi(t)-\varphi(a))^{\sigma-1}}{\Gamma(\sigma)}\, \left(\mathbf{I}_{a^+}^{1-\sigma; \, \varphi}\tilde{u}(a) +\sum_{a<t_k<t}\mathcal{J}_k(\tilde{u}(t_k^-))\right) - \mathbf{I}_{a^+}^{\varrho; \, \varphi}f(t,\tilde{u}(t))\right| \nonumber \\
&=
\left|\frac{(\varphi(t)-\varphi(a))^{\sigma-1}}{\Gamma(\sigma)}\,  \sum_{a<t_k<t} \mathcal{E}_k  +\mathbf{I}_{a^+}^{\varrho; \, \varphi}\mathcal{E}(t) \right| \nonumber \\
&\leq \frac{(\varphi(t)-\varphi(a))^{\sigma-1}}{\Gamma(\sigma)}\sum_{a<t_k<t} \epsilon\, + \epsilon\,\mathbf{I}_{a^+}^{\varrho; \, \varphi}(\varphi(t)-\varphi(a))^0 \nonumber \\
&\leq \frac{m\,\epsilon\,(\varphi(t)-\varphi(a))^{\sigma-1}}{\Gamma(\sigma)}+\frac{\epsilon\,(\varphi(t)-\varphi(a))^{\varrho}}{\Gamma(\varrho+1)}
\end{align}
Consider the impulsive $\varphi$-HFDE 
\begin{align*}
\begin{cases}
 ^H \mathbf{D}^{\varrho,\, \nu; \, \varphi}_{a^+}u(t)=f(t, u(t)),~t \in J-\{t_1, t_2,\cdots ,t_m\},\\
\Delta \mathbf{I}_{a^+}^{1-\sigma; \, \varphi}u(t_k)= \mathcal{J}_k(u(t_k^-)),  ~k = 1,2,\cdots,m, \\
 \mathbf{I}_{a^+}^{1-\sigma; \, \varphi}u(a)=\mathbf{I}_{a^+}^{1-\sigma; \, \varphi}\tilde{u}(a). 
\end{cases}
\end{align*}
Then by existence Theorem \ref{exi} it has at least one solution and  in the view of lemma \ref{fie}, we have
\begin{equation}  \label{566}
u(t) =
 \frac{(\varphi(t)-\varphi(a))^{\sigma-1}}{\Gamma(\sigma)}\, \left(\mathbf{I}_{a^+}^{1-\sigma; \, \varphi}\tilde{u}(a) +\sum_{a<t_k<t}\mathcal{J}_k(u(t_k^-))\right)+\mathbf{I}_{a^+}^{\varrho; \, \varphi}f(t,u(t)), ~\text{ $t \in J$ }.
\end{equation}
Utilizing  \eqref{56}, \eqref{566} and the hypotheses $(H_1)(ii)$ and $(H_2)(i)$, we get
\begin{align*}
&|\tilde{u}(t)-u(t)|\\
&= \left|\tilde{u}(t)-\frac{(\varphi(t)-\varphi(a))^{\sigma-1}}{\Gamma(\sigma)}\, \left(\mathbf{I}_{a^+}^{1-\sigma; \, \varphi}\tilde{u}(a) +\sum_{a<t_k<t}\mathcal{J}_k(u(t_k^-))\right)-\mathbf{I}_{a^+}^{\varrho; \, \varphi}f(t,u(t)) \right| \\
&\leq \left|\tilde{u}(t)-\frac{(\varphi(t)-\varphi(a))^{\sigma-1}}{\Gamma(\sigma)}\, \left(\mathbf{I}_{a^+}^{1-\sigma; \, \varphi}\tilde{u}(a)+\sum_{a<t_k<t}\mathcal{J}_k(\tilde{u}(t_k^-))\right) -\mathbf{I}_{a^+}^{\varrho; \, \varphi}f(t,\tilde{u}(t))\right|\\
&\qquad + \frac{(\varphi(t)-\varphi(a))^{\sigma-1}}{\Gamma(\sigma)}\sum_{a<t_k<t}\left| \mathcal{J}_k(\tilde{u}(t_k^-))-\mathcal{J}_k(u(t_k^-))\right| \\
&\qquad + \left|\mathbf{I}_{a^+}^{\varrho; \, \varphi}\left(f(t,\tilde{u}(t))-f(t,u(t)) \right)  \right| \\
&\leq \frac{m\,\epsilon\,(\varphi(t)-\varphi(a))^{\sigma-1}}{\Gamma(\sigma)}+\frac{\epsilon\,(\varphi(t)-\varphi(a))^{\varrho}}{\Gamma(\varrho+1)}\\
&\qquad + \frac{(\varphi(t)-\varphi(a))^{\sigma-1}}{\Gamma(\sigma)} \sum_{a<t_k<t}(\varphi(t_k)-\varphi(a))^{1-\sigma}\left|\tilde{u}(t_k^-)-u(t_k^-) \right|\\
&\qquad + \frac{1}{\Gamma(\varrho)} \int_{a}^{t}\varphi^{'}(s)(\varphi(t)-\varphi(s))^{\varrho-1}\,(\varphi(s)-\varphi(a))^{1-\sigma}|\tilde{u}(s)-u(s)|\, ds.
\end{align*}
Thus for any $t\in J$, we have
\begin{align*}
&(\varphi(t)-\varphi(a))^{1-\sigma}\left|\tilde{u}(t)-u(t) \right|\\
&\leq \epsilon \,\left( \frac{m}{\Gamma(\sigma)}+\frac{(\varphi(T)-\varphi(a))^{1-\sigma+\varrho}}{\Gamma(\varrho+1)}\right) +\frac{1}{\Gamma(\sigma)} \sum_{a<t_k<t}(\varphi(t_k^-)-\varphi(a))^{1-\sigma}\left|\tilde{u}(t_k^-)-u(t_k^-) \right|\\
&\qquad +\frac{(\varphi(t)-\varphi(a))^{1-\sigma}}{\Gamma(\varrho)}\int_{a}^{t}\varphi^{'}(s)(\varphi(t)-\varphi(s))^{\varrho-1}\,(\varphi(s)-\varphi(a))^{1-\sigma}|\tilde{u}(s)-u(s)|\, ds.
\end{align*}
By application of Lemma \eqref{JI} to the above inequality with
\begin{align*}
\mathcal{U}(t)&=(\varphi(t)-\varphi(a))^{1-\sigma}\left|\tilde{u}(t)-u(t) \right|,\\
\mathcal{V}(t)&= \epsilon \,\left( \frac{m}{\Gamma(\sigma)}+\frac{(\varphi(T)-\varphi(a))^{1-\sigma+\varrho}}{\Gamma(\varrho+1)}\right),\\
\mathbf{g}(t)&=\frac{(\varphi(t)-\varphi(a))^{1-\sigma}}{\Gamma(\varrho)},\\
\beta_k&=\frac{1}{\Gamma(\sigma)},
\end{align*}
 we obtain
 \begin{align*}
(\varphi(t)-\varphi(a))^{1-\sigma}\left|\tilde{u}(t)-u(t) \right|&
\leq \epsilon \,\left( \frac{m}{\Gamma(\sigma)}+\frac{(\varphi(T)-\varphi(a))^{1-\sigma+\varrho}}{\Gamma(\varrho+1)}\right)\\
&\qquad \times \left[\prod_{i=1}^{k} \left\lbrace 1+\frac{1}{\Gamma(\sigma)} E_\varrho\left( \frac{(\varphi(t)-\varphi(a))^{1-\sigma}}{\Gamma(\varrho)}\Gamma(\varrho)(\varphi(t_i)-\varphi(a))^\varrho\right) \right\rbrace  \right] \\
& \qquad \times E_\varrho\left( \frac{(\varphi(t)-\varphi(a))^{1-\sigma}}{\Gamma(\varrho)}\Gamma(\varrho)(\varphi(t)-\varphi(a))^\varrho\right). 
 \end{align*}
  Since $\varphi$ is increasing,  $(\varphi(t_i)-\varphi(a))^\varrho\leq(\varphi(t)-\varphi(a))^\varrho$, ~ $k=1,2,\cdots,m$, and hence above inequality reduces to 
 \begin{align*}
(\varphi(t)-\varphi(a))^{1-\sigma}\left|\tilde{u}(t)-u(t) \right|
 &\leq \epsilon \, \, \mathcal{A}_{m,\varrho}\,\left( \frac{m}{\Gamma(\sigma)}+\frac{(\varphi(T)-\varphi(a))^{1-\sigma+\varrho}}{\Gamma(\varrho+1)}\right) , ~t \in J,
 \end{align*}
where $\mathcal{A}_{m,\varrho}$ is defined in \eqref{32}.
Therefore, 
$$
\left\|\tilde{u}-u \right\|_{{\mathcal{PC}}_{{1-\sigma};\,\varphi}\left( J,\,\mathbb{R}\right)} \leq  C_{m,\varrho}\,\epsilon,
$$
where
$$
 C_{m,\varrho}=\mathcal{A}_{m,\varrho}\,\left( \frac{m}{\Gamma(\sigma)}+\frac{(\varphi(T)-\varphi(a))^{1-\sigma+\varrho}}{\Gamma(\varrho+1)}\right) 
 $$
This proves  the problem \eqref{p1} is UH stable.
\end{proof}
\begin{cor}
If the hypotheses $(H_1)(ii)$ and $(H_2)(i)$ are satisfied then the problem \eqref{p1} is GUH stable.
\end{cor}
\begin{proof}
Proof follows by setting $\phi_{m,\varrho}(\epsilon)=C_{m,\varrho}\,\epsilon$. 
\end{proof}

\begin{theorem} \label{ThGHU}
Suppose that $(H_1)(ii)$ and $(H_2)(i)$ hold. Moreover, assume that for a nondecreasing function $\theta \in C(J,\mathbb{R})$ there exists $\lambda_\theta > 0$ such that
$\mathbf{I}_{a^+}^{\varrho; \, \varphi} \theta(t)\leq \lambda_\theta \, \theta(t),~ t\in J$. Then, the problem \eqref{p1} is UHR stable with respect to $(\theta,\chi)$.
\end{theorem}
\begin{proof}
Let $\tilde{u}\in {\mathcal{PC}}_{{1-\sigma};\,\varphi}\left( J,\,\mathbb{R}\right) $ is a solution of the inequality \eqref{HU2}. Then proceeding as in the proof of Theorem \ref{TUH}, we obtain 
\begin{align} \label{57}
&\left|\tilde{u}(t)-\frac{(\varphi(t)-\varphi(a))^{\sigma-1}}{\Gamma(\sigma)}\, \left(u_a +\sum_{a<t_k<t}\mathcal{J}_k(\tilde{u}(t_k^-))\right) -\mathbf{I}_{a^+}^{\varrho; \, \varphi}f(t,\tilde{u}(t))\right| \nonumber \\
&\leq \frac{m\,\,\epsilon\,\chi \,(\varphi(t)-\varphi(a))^{\sigma-1}}{\Gamma(\sigma)}+ \epsilon\, \lambda_\theta \, \theta(t).
\end{align}
Utilizing\eqref{566}, \eqref{57} and the hypotheses $(H_1)(ii)$ and $(H_2)(i)$, we get 
\begin{align*}
&|\tilde{u}(t)-u(t)|\\
&= \left|\tilde{u}(t)-\frac{(\varphi(t)-\varphi(a))^{\sigma-1}}{\Gamma(\sigma)}\, \left(u_a +\sum_{a<t_k<t}\mathcal{J}_k(\tilde{u}(t_k^-))\right)-\mathbf{I}_{a^+}^{\varrho; \, \varphi}f(t,\tilde{u}(t)) \right| \\
&\qquad + \frac{(\varphi(t)-\varphi(a))^{\sigma-1}}{\Gamma(\sigma)}\sum_{a<t_k<t}\left| \mathcal{J}_k(\tilde{u}(t_k^-))-\mathcal{J}_k(u(t_k^-))\right| \\
&\qquad + \left|\mathbf{I}_{a^+}^{\varrho; \, \varphi}\left(f(t,\tilde{u}(t))-f(t,u(t)) \right)  \right|\\
&\leq \left(  \frac{m\,\,\epsilon\,\chi \,(\varphi(t)-\varphi(a))^{\sigma-1}}{\Gamma(\sigma)}+ \epsilon\, \lambda_\theta \, \theta(t)\right) \\
&\qquad + \frac{(\varphi(t)-\varphi(a))^{\sigma-1}}{\Gamma(\sigma)} \sum_{a<t_k<t}(\varphi(t_k^-)-\varphi(a))^{1-\sigma}\left|\tilde{u}(t_k^-)-u(t_k^-) \right|\\
&\qquad + \frac{1}{\Gamma(\varrho)} \int_{a}^{t}\varphi^{'}(s)(\varphi(t)-\varphi(s))^{\varrho-1}\,(\varphi(s)-\varphi(a))^{1-\sigma}|\tilde{u}(s)-u(s)|\, ds.
\end{align*}
Thus for each $t\in J,$ we obtain  
\begin{align*}
&(\varphi(t)-\varphi(a))^{1-\sigma}\left|\tilde{u}(t)-u(t) \right|\\
&\leq \epsilon\, \left(  \frac{m\,\chi}{\Gamma(\sigma)}+ \lambda_\theta \,(\varphi(t)-\varphi(a))^{1-\sigma} \theta(t)\right)\\
&\qquad + \frac{1}{\Gamma(\sigma)} \sum_{a<t_k<t}(\varphi(t_k)-\varphi(a))^{1-\sigma}\left|\tilde{u}(t_k^-)-u(t_k^-) \right|\\
&\qquad +\frac{(\varphi(t)-\varphi(a))^{1-\sigma}}{\Gamma(\varrho)}\int_{a}^{t}\varphi^{'}(s)(\varphi(t)-\varphi(s))^{\varrho-1}\,(\varphi(s)-\varphi(a))^{1-\sigma}|\tilde{u}(s)-u(s)|\, ds.
\end{align*}
By application of Lemma \eqref{JI} to the above inequality with $\mathcal{U}(t)$,
$\mathbf{g}(t)$ and $\beta_k$ as given in proof Theorem \ref{TUH} and 
$
\mathcal{V}(t)= \epsilon\, \left(  \frac{m\,\chi}{\Gamma(\sigma)}+ \lambda_\theta \,(\varphi(t)-\varphi(a))^{1-\sigma} \theta(t)\right),
$
 we obtain
 \begin{align*}
(\varphi(t)-\varphi(a))^{1-\sigma}\left|\tilde{u}(t)-u(t) \right|&
\leq \epsilon\, \left(  \frac{m\,\chi}{\Gamma(\sigma)}+ \lambda_\theta \,(\varphi(t)-\varphi(a))^{1-\sigma} \theta(t)\right) \mathcal{A}_{m,\varrho},
  \end{align*}
where $\mathcal{A}_{m,\varrho}$ is defined in \eqref{32}. Therefore
$$
(\varphi(t)-\varphi(a))^{1-\sigma}\left|\tilde{u}(t)-u(t) \right|\leq C_{m,\varrho,\theta}\, \epsilon \,(\chi+\theta(t)),
$$
where
$$
C_{m,\varrho,\theta}=\left(\frac{m}{\Gamma(\sigma)}+ \lambda_\theta \,(\varphi(T)-\varphi(a))^{1-\sigma}\right)\mathcal{A}_{m,\varrho}.
$$
Hence the problem \eqref{p1} is UHR stable.
\end{proof}
\begin{cor}
Let the hypotheses of the Theorem \ref{ThGHU} hold. Then, the problem \eqref{p1} is GHUR stable with respect to $(\theta,\chi)$.
\end{cor}
\begin{proof}
Proof follows by setting $\epsilon=1$.
\end{proof}

\section{Examples}
\begin{ex}
Consider the following impulsive $\varphi$--HFDE
\begin{align}\label{ex1}
\begin{cases}
 ^H \mathbf{D}^{\varrho,\, \nu; \, \varphi}_{0^+}u(t)=\frac{\left( \varphi(t)-\varphi(0)\right) ^{1-\sigma}}{1+|u(t)|} + 3\, \sin^{2} (\varphi(t)-\varphi(0)),~t \in J=[0,1]-\{\frac{1}{2}\}\\
\Delta \mathbf{I}_{0^+}^{1-\sigma; \, \varphi}u(\frac{1}{2})= \frac{\left( \varphi(\frac{1}{2})-\varphi(0)\right) ^{1-\sigma} \left| u({\frac{1}{2}}^{-}) \right| }{ 1+ \left| u({\frac{1}{2}}^{-}) \right| }, \\
 \mathbf{I}_{0^+}^{1-\sigma; \, \varphi}u(0)= \zeta \in \mathbb{R}.
\end{cases}
\end{align}
\end{ex}
Define 
$f: (0,1]\times {\mathbb{R}} \to \mathbb{R}$
 by 
$$
f\left(t, u \right)= \frac{\left( \varphi(t)-\varphi(0)\right) ^{1-\sigma}}{1+|u|} + 3\, \sin^{2} (\varphi(t)-\varphi(0))
$$
and 
$\mathcal{J}_1:\mathbb{R}\to \mathbb{R}$
 by
 $$
\mathcal{J}_1(u)= \frac{\left( \varphi(\frac{1}{2})-\varphi(0)\right) ^{1-\sigma} \left| u \right| }{ 1+ \left| u \right| }.
$$
Note that for any $u,v\in \R $ and $t\in [0,1]$, we have
\begin{align*}
\left|f\left(t, u \right)-f\left(t, v \right) \right| 
&\leq  \left( \varphi(t)-\varphi(0)\right) ^{1-\sigma}\, \left|\frac{1}{ 1+\left| u \right|  }-\frac{1 }{ 1+\left| v \right|  } \right| \\
& \leq \left( \varphi(t)-\varphi(0)\right) ^{1-\sigma}\,\left|\,u-v\,\right|  
\end{align*}
and
\begin{align*}
\left|\mathcal{J}_1 (u)-\mathcal{J}_1 (v) \right| & =  \left( \varphi\left( \frac{1}{2}\right)-\varphi(0)\right) ^{1-\gamma} \left(\left|\,\frac{\left| u \right| }{ 1+\left|u \right|  }-\frac{\left| v \right| }{ 1+\left| v \right|  }\, \right| \right) \leq \left(\varphi\left( \frac{1}{2}\right)-\varphi(0)\right) ^{1-\sigma} |u-v|,
\end{align*}
 \begin{align*}
\left|\mathcal{J}_1 (u) \right|=\frac{\left( \varphi\left( \frac{1}{2}\right)-\varphi(0)\right) ^{1-\sigma} \left| u \right| }{ 1+ \left| u \right| }\leq \left( \varphi\left( \frac{1}{2}\right) -\varphi(0)\right) ^{1-\sigma}=\zeta_1. 
 \end{align*}
 Observe that $f$ and $\mathcal{J}_1$ satisfy the hypotheses $(H_1)$ and $(H_2)$. By applying the Theorem \ref{exi}, the problem \eqref{ex1} has a unique solution on $[0,1]$. Also by the Theorem \ref{TUH}, the problem \eqref{ex1} is UH stable. In addition for any solution $v\in {\mathcal{PC}}_{{1-\sigma};\,\varphi}\left( J,\,\mathbb{R}\right)$ of the inequality 
 \begin{align*}
 \begin{cases}
 |^H \mathbf{D}^{\varrho,\nu;\,\varphi}_{0^+}  v(t)-f(t,v(t))|\leq \epsilon,~ t \in J \\
  |\Delta \mathbf{I}_{0^+}^{1-\sigma; \, \varphi}v(\frac{1}{2})- \mathcal{J}_1(v({\frac{1}{2}}^-))| \leq \epsilon, 
 \end{cases}
 \end{align*}
 there exists a unique solution $u$ of the problem \eqref{ex1} such that
 $$
 \left\|v-u \right\|_{{\mathcal{PC}}_{{1-\sigma};\,\varphi}\left( J,\,\mathbb{R}\right)} \leq  C_{1,\varrho}\,\epsilon,
 $$
 where
 \begin{align*}
C_{1,\varrho}&=\left( \frac{1}{\Gamma(\sigma)}+\frac{(\varphi(1)-\varphi(0))^{1-\sigma+\varrho}}{\Gamma(\varrho+1)}\right) \left(1+\frac{1}{\Gamma(\sigma)}E_\varrho ((\varphi(1)-\varphi(0))^{1-\sigma+\varrho}) \right)\\
     & \qquad \times E_\varrho\left((\varphi(1)-\varphi(0))^{1-\sigma+\varrho} \right).
 \end{align*}
\section*{Acknowledgment}
The first author  acknowledges the Science and Engineering Research Board (SERB), New Delhi, India for the Research Grant (Ref: File no. EEQ/2018/000407).


\begin{thebibliography}{99}
\bibitem{fec} M. Feckan, Y. Zhou, J. Wang, On the concept and existence of solution for impulsive fractional differential equations, Commun Nonlinear Sci Numer Simulat
, 17(7) (2012) 3050--3060.

\bibitem{WAZN}  
G. Wang, B. Ahmad, L. Zhang,J. J. Nieto,   
Comments on the concept of existence of solution for impulsive
fractional differential equations,
Commun Nonlinear Sci Numer Simulat 19 (2014) 401-403. 

\bibitem{Wang1} 
 J. Wang, Y. Zhou, Z. Lin,
  On a new class of impulsive fractional differential equations,
  Applied Mathematics and Computation 242 (2014), 649--657.
  
\bibitem{Wang}J.Wang, Y. Zhou, M. Fe$\breve{c}$kan, On recent developments in the theory of boundary value problems for impulsive fractional differential equations,Computers and Mathematics with Applications 64 (2012), 3008-–3020.  

 \bibitem{Wang22} 
   J. Wang, Y. Zhou, M. Fe$\breve{c}$kan,
  On recent developments in the theory of boundary value problems for
  impulsive fractional differential equations,
  Computers and Mathematics with Applications 64 (2012), 3008--3020.
  
\bibitem{Benchohra1} 
  M. Benchohra, F. Berhoun,
  Impulsive fractional differential equations with variable times,
  Computers and Mathematics with Applications 59 (2010), 1245--1252 .
  
 \bibitem{Mophou} 
  G. M. Mophou,
  Existence and uniqueness of mild solutions to impulsive fractional
  differential equations,
  Nonlinear Analysis 72 (2010), 1604--1615. 
  
\bibitem{Zhang}
  L. Zhang, G. Wang,
  Existence of solutions for nonlinear fractional
  differential equations with impulses and anti-periodic
  boundary conditions,
  Electronic Journal of Qualitative Theory of Differential Equations
  7(2011), 1--11. 
  
  \bibitem{Ahmad} 
    B. Ahmad, S. Sivasundaram, Existence of solutions for impulsive integral boundary value problems of fractional order.Nonlinear Analysis: Hybrid Systems 4 (2010), 134--141.     
  
  \bibitem{Liu1}
    Z. Liu , X.Li,
    Existence and uniqueness of solutions for the nonlinear impulsive fractional differential equations, 
     Commun Nonlinear Sci Numer Simulat 18 (2013),1362-–1373. 
  
 \bibitem{Feckan} 
    M. Fe$\breve{c}$kan , Y. Zhou , J. Wang,
    On the concept and existence of solution for impulsive fractional
    differential equations,
    Commun Nonlinear Sci Numer Simulat 17 (2012), 3050--3060 . 
  
   \bibitem{Ali}
       A. Ali, K. Shah, D. Baleanu,
       Ulam stability results to a class of
       nonlinear implicit boundary value problems
       of impulsive fractional differential equations,
       Advances in Difference Equations (2019) 2019:5.
       
 \bibitem{Wang11} J.Wang, Y. Zhou, M. Fe$\breve{c}$kan,  Nonlinear impulsive problems for fractional differential equations and Ulam stability, Computers and Mathematics with Applications 64 (2012), 3389-–3405.       
  
\bibitem{Wang2} J.Wang, M. Fe$\breve{c}$kan, Y. Zhou,   Ulam’s type stability of impulsive ordinary differential equations, Journal of Mathematical Analysis and
 Applications 395 (2012), 258--264.  
  
\bibitem{Benchohra}M. Benchohra, B. A. Slimani, Existence and Uniqueness of Solutions to Impulsive Fractional Differential Equations ,Electronic Journal of Differential Equations,2009(2009), No. 10, pp. 1–11.


 \bibitem{Benchohra2}
  M. Benchohra, D. Seba,
  Impulsive fractional differential equations in banach spaces,
  Electronic Journal of Qualitative Theory of Differential Equations
  Spec. Ed. I, 8(2009), 1--14. 
 
 \bibitem{Benchohra3}
   M. Benchohra, B. Slimani,
   Existence and uniqueness of solutions to impulsive fractional differential equations,
   Electronic Journal of Differential Equations
    2009(10)(2009), 1--11.

\bibitem{AS} B. Ahmad , S. Sivasundaram,  
Existence results for nonlinear impulsive hybrid boundary value problems involving fractional differential equations.
Nonlinear Anal Hybrid Syst 2009;3:251–8.

\bibitem{BS} M. Benchohra ,D. Seba, 
Impulsive fractional differential equations in Banach spaces. Electron J Qual Theory Differ Equ 2009 [Special Edition I, No. 8, 14
pp].

\bibitem{BK} K. Balachandran ,S. Kiruthika,  
Existence of solutions of abstract fractional impulsive semilinear evolution equations. Electron J Qual Theory Differ Equ
2010 [No. 4, 12 pp].

\bibitem{WAZ} G. Wang , B. Ahmad , L. Zhang,
  Impulsive anti-periodic boundary value problem for nonlinear differential equations of fractional order. Nonlinear Anal
2011;74:792–804.

\bibitem{Harrat}	A. Harrat, J. J. Nieto, A. Debbouche,
Solvability and optimal controls of impulsive Hilfer fractional delay evolution inclusions with Clarke subdifferential. J. Computational Applied Mathematics 344: 725-737 (2018).

\bibitem{Harikrishnan} S. Harikrishnan , K. Kanagarajan, S. Sivasundaram,
Stability analysis and dynamics of impulsive differential equations under Hilfer fractional derivative. Harikrishnan, S.,
et al. 2018 Nonlinear Studies 25, 403-415.  
  
\bibitem{Ahmed} H. M. Ahmed, M. M. El-Borai, H. M. El-Owaidy,  A. S. Ghanem,
Impulsive Hilfer fractional differential
Equations,  Advances in Difference Equations, 
Ahmed et al. Advances in Difference Equations  (2018) 2018:226 
https://doi.org/10.1186/s13662-018-1679-7.  
  
\bibitem{Fernandez}
A. Fernandez, M. Özarslan, D. Baleanu,
On fractional calculus with general analytic kernels,
Applied Mathematics and Computation, 354 (2019) 248–265.  
  
\bibitem{Kilbas} A. A. Kilbas, H. M. Srivastava, J. J. Trujillo, 
 Theory and applications of fractional differential equations,
 North--Holland Mathematics Studies, Elsevier, Amsterdam, Vol. 207, 2006.
 
 \bibitem{Almeida} R. Almeida, 
  A Caputo fractional derivative of a function with respect to another function ,
  Commun. Nonlinear Sci. Numer. Simulat., 44, 460--481 (2017).   
  
\bibitem{Ameen}
 R. Ameen, F. Jarad, T. Abdeljawad, 
 Ulam Stability for Delay Fractional Differential Equations with a Generalized Caputo Dervative,
 Filomat, 32(15) (2018) 5265--5274.

\bibitem{Jarad}
 F. Jarad, S. Harikrishnan, K. Shah, K. Kanagarajan,
 Existence and stability results to a class of fractional random implicit differential equations involving a generalized Hilfer fractional derivative,
 Discrete and continuous dynamical systems series S,(2018)209--219.

\bibitem{Jarad1}
F. Jarad, T. Abdeljawad,
Generalized fractional derivatives and Laplace transform,
Discrete and continuous dynamical systems series S,(2019)1775--1786. 

\bibitem{Sousa1}J.V.C. Sousa, Oliveira E. Capelas de., On the  $\varphi$--Hilfer fractional derivative. Commun.Nonlinear Sci. Numer. Simulat. 60(2018),72–-91.


 \bibitem{Sousa2} J.V.C. Sousa, Oliveira E. Capelas de., A Gronwall inequality and the Cauchy-type problem by means of $\varphi$-
  operator, arXiv:1709.03634, (2017). 

 \bibitem{Hilfer} R. Hilfer, 
  Applications of fractional calculus in Physics, 
  World Scientific, Singapore, 2000.

\bibitem{JVCO} J.V.C.  Sousa, E. Capelas De Oliveira, 
A Gronwall inequality and the cauchy–type
problem by means of $\varphi$–-Hilfer operator,
arXiv:1709.03634(2017).


\bibitem{KKAM} K.D. Kucche, A. D. Mali, J.V.C. Sousa, 
On the nonlinear $\varphi$-Hilfer fractional differential equations
Comp. Appl. Math. (2019) 38: 73. https://doi.org/10.1007/s40314-019-0833-5.

\bibitem{KKJK} 
  K.D. Kucche, J.P. Kharade, J.V.C. Sousa, On the Nonlinear Impulsive $\varphi$-Hilfer Fractional Differential Equations, arXiv:1901.01814, (2019).

\bibitem{SousaKucche} J.V.C. Sousa, K. D. Kucche, E. Capelas de Oliveira, 
Stability of $ \varphi $-Hilfer impulsive fractional differential equations,
Applied Mathematics Letters 88 (2019) 73–80.

\bibitem{Liu} K Liu, J. Wang and, D O’Regan,
Ulam-Hyers-Mittag-Leffler stability for $ \varphi $-Hilfer fractional-order delay differential equations, 
Advances in Difference Equations, (2019) 2019:50
https://doi.org/10.1186/s13662-019-1997-4.

\bibitem{Bai}Z. Bai, X. Dong, C. Yin,
  Existence results for impulsive nonlinear fractional differential equation with mixed boundary conditions,Bai et al. Boundary Value Problems  (2016) 2016:63.

\bibitem{zhou} Zhou Y.,Basic theory of fractional differential equations. World scientific, 2014.

 \bibitem{Jose} 
  J.V.C. Sousa, D.S. Oliveira, E. Capelas de Oliveira,
  A note on the mild solutions of Hilfer impulsive fractional differential equations, arXiv:1811.09256 (2018).

\end{thebibliography}
\end{document}